\documentclass[a4paper, 11pt]{article}
\usepackage{amsmath}
\usepackage{amssymb,esint}
\usepackage{amscd}
\usepackage{xspace}
\usepackage{fancyhdr}
\usepackage{color}
\usepackage{srcltx} 

\newtheorem{theorem}{Theorem}[section]

\newtheorem{corollary}[theorem]{Corollary}

\newtheorem{definition}[theorem]{Definition}

\newtheorem{lemma}[theorem]{Lemma}

\newtheorem{proposition}[theorem]{Proposition}
\newtheorem{remark}[theorem]{Remark}

\newenvironment{proof}[1][Proof]{\textbf{#1.} }{\hfill\rule{0.5em}{0.5em}}
{\catcode`\@=11\global\let\AddToReset=\@addtoreset
\AddToReset{equation}{section}

\AddToReset{theorem}{section}

\def\nc{\newcommand}

\def\om{\omega}

 \def\Om{\Omega}

\nc\pa{\partial}

\nc\CC{\mathbb{C}}
\nc\RR{\mathbb{R}}
\nc\QQ{\mathbb{Q}}
\nc\ZZ{\mathbb{Z}}
\nc\NN{\mathbb{N}}

\begin{document}
\title{Good-$\lambda$  and Muckenhoupt-Wheeden type bounds in quasilinear measure datum problems, with applications}
	\author{
	{\bf Quoc-Hung Nguyen\thanks{E-mail address: quochung.nguyen@sns.it, Scuola Normale Superiore, Centro Ennio de Giorgi, Piazza dei Cavalieri 3, I-56100
			Pisa, Italy.}~~and~Nguyen Cong Phuc\thanks{E-mail address: pcnguyen@math.lsu.edu, Department of Mathematics, Louisiana State University, 303 Lockett Hall,
			Baton Rouge, LA 70803, USA. }}}
\date{}  
\maketitle
\begin{abstract}
Weighted good-$\lambda$ type inequalities  and  Muckenhoupt-Wheeden type bounds 
are obtained for gradients of solutions to a class of  quasilinear elliptic equations
with measure data. 
Such results are obtained globally  over sufficiently flat domains in $\RR^n$ in the sense of Reifenberg. The principal operator here is modeled after the $p$-Laplacian, where for the first time singular case $\frac{3n-2}{2n-1}<p\leq 2-\frac{1}{n}$ is considered.
Those bounds lead to
useful compactness criteria for  solution sets of quasilinear elliptic equations
with measure data. As an application, sharp existence results and sharp bounds on the size of removable singular sets  
are deduced for a quasilinear Riccati type equation having a gradient source term with linear or super-linear power  growth.

\medskip

\medskip

\medskip


\medskip

\end{abstract}   
                  
\tableofcontents
									
 \section{Introduction and main results} 
 In this article, we are concerned with  global weighted gradient estimates for quasilinear elliptic equations
with measure data. Such estimates are then applied to address the question of sharp existence and removable singularities  for a 
quasilinear equation with strong power growth in the gradient known as an  equation of Riccati type.

In particular, our first goal is to obtain `good-$\lambda$' type bounds   and  nonlinear 
Muckenhoupt-Wheeden type inequalities  for gradients of  solutions to  quasilinear elliptic equations with measure data:   
\begin{eqnarray}\label{5hh070120148}
\left\{ \begin{array}{rcl}
 -{\rm div}(A(x, \nabla u))&=& \mu \quad \text{in} ~\Omega, \\
u&=&0  \quad \text{on}~ \partial \Omega.
\end{array}\right.
\end{eqnarray} 																																						
Here  $\Omega$ is a bounded open subset of $\mathbb{R}^{n}$, $n\geq2$, and $\mu$ is  a finite signed Radon  in $\Omega$.

Our second goal is to employ those estimates to study  a quasilinear Riccati type equation with measure data:																																					
\begin{eqnarray}\label{Riccati}
\left\{ \begin{array}{rcl}
 -{\rm div}(A(x, \nabla u))&=& |\nabla u|^q + \mu \quad \text{in} ~\Omega, \\
u&=&0  \quad ~~~~~~~~~~~\text{on}~ \partial \Omega,
\end{array}\right.
\end{eqnarray} 																																						
 and removable singularities for related  `homogeneous' equations 
$-\operatorname{div}(A(x,\nabla u))=|\nabla u|^q$, $q>0$.
In particular, we address a  question of sharp existence for \eqref{Riccati} posed by Igor E. Verbitsky (personal communication), which has also been  stated as an open problem  in \cite{VHV},
pages 13--14.

In \eqref{5hh070120148}-\eqref{Riccati} and  throughout the paper, the nonlinearity  $A:\mathbb{R}^n\times \mathbb{R}^n\to \mathbb{R}^n$ is a Carath\'eodory vector valued function, i.e., $A(x,\xi)$ is measurable in $x$ and continuous with respect to $\xi$ for a.e. $x$. Moreover, for a.e. $x$, $A(x,\xi)$ is continuously differentiable 
in $\xi$ away from the origin and																			satisfies 
                                       \begin{align}
                                       \label{condi1}| A(x,\xi)|\le \Lambda |\xi|^{p-1},\quad | \nabla_\xi A(x,\xi)|\le \Lambda |\xi|^{p-2},
                                       \end{align}
                                        \begin{align}
                                       \label{condi2}  \langle \nabla_\xi A(x,\xi)\eta,\eta\rangle\geq \Lambda^{-1} |\eta|^2|\xi|^{p-2},
                                       \end{align}
                                          for every $(\xi,\eta)\in \mathbb{R}^n\times \mathbb{R}^n\backslash\{(0,0)\}$ and a.e. $x\in \mathbb{R}^n$, where  $\Lambda$ is a  positive constant.

																					As for $p$ in \eqref{condi1}-\eqref{condi2},  in this paper, we shall restrict ourselves to the `singular' case: 
                                          \begin{align*}
                                          \frac{3n-2}{2n-1}<p\leq 2-\frac{1}{n}.
                                          \end{align*}        
However, as we remark later, all of the results obtained in this paper also hold in the `regular' case $2-\frac{1}{n}<p\leq n$.
It also makes sense to consider the case $1<p\leq \frac{3n-2}{2n-1}$. Unfortunately, our method breaks down in this case. However, some useful partial results could be obtained for this range of $p$, and they  will be presented elsewhere.

For our purpose  we also require that the nonlinearity $A$ satisfy a smallness condition of BMO type in the $x$-variable. We call such a condition the $(\delta, R_0)$-BMO condition.

\begin{definition} 
 We say that $A({x, \zeta})$ satisfies a $(\delta, R_0)$-BMO condition for some $\delta, R_0>0$  if
\begin{equation*}
[A]_{R_0}:=\mathop {\sup }\limits_{y\in \mathbb{R}^n,0<r\leq R_0}\fint_{B_r(y)}\Theta(A,B_r(y))(x)dx \leq \delta,
\end{equation*}         
where 
\begin{equation*}
\Theta(A,B_r(y))(x):=\mathop {\sup }\limits_{\zeta\in\mathbb{R}^n\backslash\{0\}}\frac{|A(x,\zeta)-\overline{A}_{B_r(y)}(\zeta)|}{|\zeta|^{p-1}},
\end{equation*}
and $\overline{A}_{B_r(y)}(\zeta)$ denotes the average of $A(\cdot,\zeta)$ over the ball $B_r(y)$, i.e.,
\begin{equation*}
\overline{A}_{B_r(y)}(\zeta):=\fint_{B_r(y)}A(x,\zeta)dx=\frac{1}{|B_r(y)|}\int_{B_r(y)}A(x,\zeta)dx.
\end{equation*}                                                              
\end{definition}

A typical example of such a nonlinearity $A$ is given by $A(x, \xi)=|\xi|^{p-2}\xi$ which gives rise to the standard $p$-Laplacian $\Delta_p u={\rm div}(|\nabla u|^{p-2}\nabla u)$. 

In the case  $p=2$, the above $(\delta, R_0)$-BMO condition 
was introduced in \cite{BW2}, whereas  such a condition for  general $p\in(1,\infty)$  appears  in the  paper \cite{Ph3}.
We remark that the  $(\delta, R_0)$-BMO condition allows $A(x, \xi)$ has discontinuity in $x$ and it can be used as an 
appropriate substitute for the Sarason \cite{Sa} VMO  condition.

Due to the global nature of our gradient estimates, we also require certain regularity on the ground domain $\Omega$.
 Namely, at each
boundary point and every scale, we ask that the boundary of $\Omega$  be trapped between two hyperplanes separated by a distance that
depends on the scale.    The following  defines the relevant geometry  precisely.
\begin{definition}
 Given $\delta\in (0, 1)$ and $R_0>0$, we say that $\Omega$ is a $(\delta, R_0)$-Reifenberg flat domain if for every $x\in \partial \Omega$
 and every $r\in (0, R_0]$, there exists a
 system of coordinates $\{ z_{1}, z_{2}, \dots,z_{n}\}$,
 which may depend on $r$ and $x$, so that  in this coordinate system $x=0$ and that
\[
B_{r}(0)\cap \{z_{n}> \delta r \} \subset B_{r}(0)\cap \Omega \subset B_{r}(0)\cap \{z_{n} > -\delta r\}.
\]
\end{definition}

For more properties of Reifenberg flat domains and their many applications, we refer to the papers \cite{HM, Jon, KT1, KT2, 
55Re, Tor}. This class of domains appeared first in a paper of Reifenberg (see \cite{55Re}) in the context of the Plateau problem. Here we remark that they include $C^1$ domains and  Lipschitz domains with sufficiently small Lipschitz constants (see \cite{Tor}).  Moreover, they also include certain domains with fractal boundaries and thus provide a wide range of applications.

  In this paper, all solutions of   \eqref{5hh070120148} and \eqref{Riccati} with a finite signed measure $\mu$ in $\Omega$ will be understood in the     renormalized sense (see \cite{11DMOP}). For $\mu\in\mathfrak{M}_b(\Omega)$ (the set of finite signed measures in $\Omega$), we will tacitly extend it by zero to $\Omega^c:=\mathbb{R}^n\setminus\Omega$.	We let $\mu^+$ and $\mu^-$  be the positive and negative parts, respectively, of a measure $\mu\in\mathfrak{M}_b(\Omega)$. We denote by $\mathfrak{M}_0(\Omega)$ the space of finite signed measures in $\Omega$ which are absolutely continuous with respect to the capacity $c^{\Omega}_{1,p}$. Here  $c^{\Omega}_{1,p}$ is the $p$-capacity defined  for each compact set $K\subset\Omega$ by
  \begin{equation*}
  c^{\Omega}_{1,p}(K)=\inf\left\{\int_{\Omega}{}|{\nabla \varphi}|^pdx:\varphi\geq \chi_K,\varphi\in C^\infty_c(\Omega)\right\},
  \end{equation*}
	where $\chi_{K}$ is the characteristic function of the set $K$.
	  We also denote by $\mathfrak{M}_s(\Omega)$ the space of finite signed measures in $\Omega$ with support on a set of zero $c^{\Omega}_{1,p}$-capacity. It is known that  any $\mu\in\mathfrak{M}_b(\Omega)$ can be written  uniquely  in the form $\mu=\mu_0+\mu_s$ where $\mu_0\in \mathfrak{M}_0(\Omega)$ and $\mu_s\in \mathfrak{M}_s(\Omega)$ (see \cite{FS}).
  It is also known  that any  $\mu_0\in \mathfrak{M}_0(\Omega)$ can be written in the form $\mu_0=f-\operatorname{div}(F)$ where $f\in L^1(\Omega)$ and $F\in L^{\frac{p}{p-1}}(\Omega,\mathbb{R}^n)$.
  
  For $k>0$, we define the usual  two-sided truncation operator $T_k$ by
	$$T_k(s)=\max\{\min\{s,k\},-k\}, \qquad s\in\mathbb{R}.$$ 
	
	For our purpose, the following notion of gradient is needed.
	If $u$ is a measurable function defined  in $\Omega$, finite a.e., such that $T_k(u)\in W^{1,p}_{loc}(\Omega)$ for any $k>0$, then there exists a measurable function $v:\Omega\to \mathbb{R}^n$ such that $\nabla T_k(u)= v \chi_{\{|u|< k\}} $ 
  a.e. in $\Omega$  for all $k>0$ (see \cite[Lemma 2.1]{bebo}). In this case, we define the gradient $\nabla u$ of $u$ by $\nabla u:=v$. It is known that  $v\in L^1_{loc}(\Omega, \mathbb{R}^n)$ if and only if  $u\in W^{1,1}_{loc}(\Omega)$ and then $v$ is the usual weak gradient of $u$. On the other hand, for $1<p\leq 2-\frac{1}{n}$, by looking at the fundamental solution we see that in general distributional solutions of \eqref{5hh070120148}  may not even belong to $u\in W^{1,1}_{loc}(\Omega)$.
	
	The notion of renormalized solutions is a generalization of that of entropy solutions introduced in \cite{bebo} and \cite{BGO}, where the right-hand side is assumed to be  in $L^1(\Omega)$ or in  $\mathfrak{M}_{0}(\Omega)$.
Several equivalent definitions of renormalized solutions
were  given  in \cite{11DMOP}. Here we use the following one:

  
  \begin{definition} \label{derenormalized} Let $\mu=\mu_0+\mu_s\in\mathfrak{M}_b(\Omega)$, with $\mu_0\in \mathfrak{M}_0(\Omega)$ and $\mu_s\in \mathfrak{M}_s(\Omega)$. A measurable  function $u$ defined in $\Omega$ and finite a.e. is called a renormalized solution of \eqref{5hh070120148}
  	if $T_k(u)\in W^{1,p}_0(\Omega)$ for any $k>0$, $|{\nabla u}|^{p-1}\in L^r(\Omega)$ for any $0<r<\frac{n}{n-1}$, and $u$ has the following additional property. For any $k>0$ there exist  nonnegative Radon measures $\lambda_k^+, \lambda_k^- \in\mathfrak{M}_0(\Omega)$ concentrated on the sets $\{u=k\}$ and $\{u=-k\}$, respectively, such that 
  	$\mu_k^+\rightarrow\mu_s^+$, $\mu_k^-\rightarrow\mu_s^-$ in the narrow topology of measures and  that
  	\[
  	\int_{\{|u|<k\}}\langle A(x,\nabla u),\nabla \varphi\rangle
  	dx=\int_{\{|u|<k\}}{\varphi d}{\mu_{0}}+\int_{\Omega}\varphi d\lambda_{k}%
  	^{+}-\int_{\Omega}\varphi d\lambda_{k}^{-},
  	\]
  	for every $\varphi\in W^{1,p}_0(\Omega)\cap L^{\infty}(\Omega)$.
  	  \end{definition}

Here we recall that a sequence  $\{\mu_{k}\} \subset
\mathfrak{M}_{b}(\Omega)$ is said to converge in the narrow topology of measures to $\mu \in
\mathfrak{M}_{b}(\Omega)$ if
$$\lim_{k\rightarrow\infty}\int_{\Omega}\varphi \, d\mu_{k}=\int_{\Omega}\varphi \,
d\mu,$$
for every bounded and continuous function $\varphi$ on $\Omega$.

It is known that if $\mu\in \mathfrak{M}_0(\Omega)$ then there is one and only one  renormalized solution of 
\eqref{5hh070120148} (see \cite{BGO, 11DMOP}). However, to the best of our knowledge, for a general $\mu\in \mathfrak{M}_b(\Omega)$ the uniqueness of renormalized solutions of \eqref{5hh070120148} is still an open problem.

In the first main result of the paper,  we are  concerned with a nonlinear Muckenhoupt and Wheeden type  bound for gradients of solutions of 
\eqref{5hh070120148} that involves the class of $\mathbf{A}_{\infty}$ weights.
We recall that   a positive function $w\in L^1_{\text{loc}}(\mathbb{R}^{n})$ is said to be an $\mathbf{A}_{\infty}$ weight if there are two positive constants $C$ and $\nu$ such that
$$w(E)\le C \left(\frac{|E|}{|B|}\right)^\nu w(B),
$$
for all balls $B=B_\rho(x)$ and all measurable subsets $E$ of $B$. The pair $(C,\nu) $ is called the $\mathbf{A}_\infty$ constants of $w$ and is denoted by $[w]_{\mathbf{A}_\infty}$.

\begin{theorem} \label{101120143-p} Let $\mu\in \mathfrak{M}_b(\Omega)$ and $\frac{3n-2}{2n-1}<p\leq 2-\frac{1}{n}$.  Let $\Phi: [0,\infty)\rightarrow [0,\infty)$ be a strictly increasing function such that $\Phi(0)=0$, $\lim_{t\rightarrow\infty}\Phi(t)=\infty$, and $\Phi$ is of moderate growth, i.e., $\Phi(2t)\leq c\, \Phi(t)$ for all $t\geq 0$ with a constant $c>1$. 
For any $w\in \mathbf{A}_{\infty}$, we can find  $\delta=\delta(n,p,\Lambda, \Phi, [w]_{\mathbf{A}_{\infty}})\in (0,1)$ such that if $\Omega$ is  $(\delta,R_0)$-Reifenberg flat  and $[A]_{R_0}\le \delta$ with some $R_0>0$ then  for any renormalized solution $u$ of \eqref{5hh070120148} we have                            
     \begin{equation}\label{101120144}
\int_{\Omega}\Phi(|\nabla u|) w(x)dx \leq C \int_{\Omega} \Phi([\mathbf{M}_1(\mu)]^{\frac{1}{p-1}}) w(x)dx.
    \end{equation} 
                                        Here $C$ depends only  on $n,p,\Lambda, \Phi, [w]_{\mathbf{A}_\infty}$, and $diam(\Omega)/R_0$.
   \end{theorem}

In \eqref{101120144} and in what follows the operator  $\mathbf{M}_1$ is the first order fractional maximal function defined by
                                          \begin{align*}
                                          \mathbf{M}_1(\mu)(x):=\sup_{\rho>0}\frac{|\mu|(B_\rho(x))}{\rho^{n-1}}~~\forall  x\in \mathbb{R}^n.
                                          \end{align*}

We shall also use the  Hardy-Littlewood maximal function ${\bf M}$ defined for each locally integrable function  $f$ in $\mathbb{R}^{n}$ by
\begin{equation*}
{\bf M}(f)(x)=\sup_{\rho>0}\fint_{B_\rho(x)}|f(y)|dy~~\forall x\in\mathbb{R}^{n}.
\end{equation*}

The proof of Theorem \ref{101120143-p} above is in fact a consequence of the following good-$\lambda$ type inequality involving both $\mathbf{M}_1$ and ${\bf M}$, which is interesting in its own right.																					
\begin{theorem}\label{5hh23101312}   Let $w\in {\bf A}_\infty$, $\mu\in\mathfrak{M}_b(\Omega)$, and $\frac{3n-2}{2n-1}<p\leq 2-\frac{1}{n}$.   For any $\varepsilon>0,R_0>0$ one can find  constants $\delta_1=\delta_1(n,p,\Lambda,\varepsilon,[w]_{{\bf A}_\infty})\in (0,1)$, $\delta_2=\delta_2(n,p,\Lambda,\varepsilon,[w]_{{\bf A}_\infty},diam(\Omega)/R_0)\in (0,1)$ and $\Lambda_0=\Lambda_0(n,p,\Lambda)>1$ such that if $\Omega$ is a $(\delta_1,R_0)$-Reifenberg flat domain and $[A]_{R_0}\le \delta_1$ then  for any renormalized solution $u$ to \eqref{5hh070120148}, we have 
\begin{align*}
&w(\{( {\bf M}(|\nabla u|^{\gamma_0}))^{1/\gamma_0}>\Lambda_0\lambda, (\mathbf{M}_1(\mu))^{\frac{1}{p-1}}\le \delta_2\lambda \}\cap \Omega) \nonumber\\
&\qquad\leq C\varepsilon w(\{ ({\bf M}(|\nabla u|^{\gamma_0}))^{1/\gamma_0}> \lambda\}\cap \Omega), 
\end{align*}
for any $\lambda>0$. Here  $\gamma_0$ is a number in $\left(\frac{2-p}{2}, \frac{(p-1)n}{n-1}\right)$, and the constant $C$  depends only on $n,p,\Lambda,diam(\Omega)/R_0$, and $[w]_{{\bf A}_\infty}$. 
\end{theorem}

We now  have some comments on the proof of Theorem \ref{5hh23101312}.  It is based on various tools developed for quasilinear equations with measure data and linear or nonlinear potential and Calder\'on-Zygmund theories (see, e.g., \cite{bebo, BW1, 11DMOP, 55DuzaMing, Duzamin2, 55MePh2, Mi2,55Mi0, 55QH2, 55Ph0, 55Ph2}). The key ingredients in this work which make it possible for us to apply those tools are  some new local comparison estimates obtained in the singular case $\frac{3n-2}{2n-1}<p\leq 2-\frac{1}{n}$; see Lemmas \ref{111120149} and \ref{111120149"} below. Earlier those comparison estimates were known in the case  $p>2-\frac{1}{n}$ (see \cite{Duzamin2, Mi2}), and thus in fact one can follow the method of this paper to prove Theorem  
\ref{5hh23101312} in the case $p>2-\frac{1}{n}$ (with $\gamma_0=1$). With this remark, Theorem 
\ref{101120143-p} also holds for $p>2-\frac{1}{n}$ and so does its Corollaries \ref{compactness}-\ref{MW-2} below. 
It is worth mentioning that the comparison estimates obtained in Lemma \ref{111120149} can also be used to extend the recent gradient pointwise estimates by potentials obtained in \cite{Duzamin2} (see also \cite{55DuzaMing, KuMi}) to the case $\frac{3n-2}{2n-1}<p\leq 2-\frac{1}{n}$. This will be pursued in our forthcoming work.

We remark that the function $\Phi$ in Theorem \ref{101120143-p} is quite general. In particular, we do not ask
$\Phi$ to be convex or to satisfy the so-called $\nabla_2$ condition: $\Phi(t)\geq \frac{1}{2a}\Phi(at)$ for some $a>1$
and for all $t\geq 0$. As such one can take, e.g., $\Phi(t)=t^q$ for any $q>0$, or even $\Phi(t)=[\log(1+t)]^{\alpha}$, $\alpha>0$, etc. We emphasize that the introduction of $\Phi$
in Theorem \ref{101120143-p} is not just for the sake of generality. In fact, such $\Phi$ will serve as an indispensable tool in our study of the Riccati type equation \eqref{Riccati}. In particular, Theorem \ref{101120143-p} with such general $\Phi$ is needed  to obtain a useful criterion for compactness of solution sets of 
equation \eqref{5hh070120148}; see Corollary \ref{compactness} below.

In the case $\Phi(t)=t^q$, $q>0$, estimates of the form \eqref{101120144} were obtained for (linear) fractional integral operators by Muckenhoupt-Wheeden in the pioneering work \cite{MW}.  It is worth mentioning that for quasilinear problems the fractional maximal operator approach  has been introduced in Mingione \cite{55Mi0}.
Also, for $\Phi(t)=t^q$, $q>0$, and for $p>2-\frac{1}{n}$  estimate
\eqref{101120144} was  obtained in \cite{55Ph2}. Thus Theorem \ref{101120143-p} is new at least in the case $\frac{3n-2}{2n-1}<p\leq 2-\frac{1}{n}$ considered in this paper. Moreover, using Theorem \ref{5hh23101312} one can also obtain a weighted Lorentz space estimate in the spirit of  \cite{55Ph2} but now for the singular case $\frac{3n-2}{2n-1}<p\leq 2-\frac{1}{n}$. For a weight function
$w$, the weighted Lorentz space $L^{q,s}_w(\Omega)$, 
$q\in(0,\infty), s\in(0,\infty]$,  is the space of measurable functions $g$ on $\Omega$ such that 
\begin{equation*}
\|g\|_{L^{q,s}_w(E)}:=\left\{ \begin{array}{l}
\left(q\int_{0}^{\infty}\left(\rho^qw\left(\{x\in \Omega:|g(x)|>\rho\}\right)\right)^{\frac{s}{q}}\frac{d\rho}{\rho}\right)^{1/s}<\infty~\text{ if }~s<\infty, \\ 
\sup_{\rho>0}\rho \left( w\left(\{x\in \Omega:|g(x)|>\rho\}\right)\right)^{1/q}<\infty~~\text{ if }~s=\infty. \\ 
\end{array} \right.
\end{equation*}              
Here we write $w(E)=\int_{E}w(x)dx$ for a measurable set $E\subset \mathbb{R}^{n}$.  Obviously,  
$
\|g\|_{L^{q,q}_w(\Omega)}=\|g\|_{L^q_w(\Omega)}
$,
thus $L^{q,q}_w(\Omega)=L^{q}_w(\Omega)$.               As usual, when $w \equiv 1$  we  write  $L^{q,s}(\Omega)$ instead of $L^{q,s}_w(\Omega)$.

 \begin{theorem} \label{101120143} Let $\mu\in \mathfrak{M}_b(\Omega)$ and $\frac{3n-2}{2n-1}<p\leq 2-\frac{1}{n}$.   For any $w\in \mathbf{A}_{\infty}$, $0< q<\infty$, $0<s\leq\infty$ we can find  $\delta=\delta(n,p,\Lambda, q,s, [w]_{\mathbf{A}_{\infty}})\in (0,1)$ such that if $\Omega$ is  $(\delta,R_0)$-Reifenberg flat   and $[A]_{R_0}\le \delta$ for some $R_0>0$, then  for any renormalized solution $u$ of \eqref{5hh070120148}, we have                            
                  \begin{equation*}
                              \|\nabla u\|_{L^{q,s}_w(\Omega)}\leq C \|[\mathbf{M}_1(\mu)]^{\frac{1}{p-1}}\|_{L^{q,s}_w(\Omega)}.
                                       \end{equation*} 
                                        Here the constant $C$ depends only  on $n,p,\Lambda,q,s, [w]_{\mathbf{A}_\infty}$ and $diam(\Omega)/R_0$.               
\end{theorem}

Theorem \ref{101120143-p} implies the following compactness criterion for solution sets of equation  \eqref{5hh070120148}. This result will be needed in the proof of Theorem \ref{main-Ric} below.

\begin{corollary}\label{compactness} Suppose that  $\frac{3n-2}{2n-1}<p\leq 2-\frac{1}{n}$. For each $j>0$, let $\mu_j \in \mathfrak{M}_b(\Omega)$. 
Let $u_j$ be a solution of \eqref{5hh070120148} with datum $\mu=\mu_j$ in $\Omega$.
Assume that 
$\{\left[\mathbf{M}_1(\mu_j)\right]^{\frac{q}{p-1}}\}_{j}$, $q>0$, is a bounded and equi-integrable subset of  $L^1_w(\Omega)$
for some   $w\in \mathbf{A}_{\infty}$.
Then, there exists  $\delta=\delta(n,p,\Lambda, q, [w]_{\mathbf{A}_{\infty}})\in (0,1)$ such that if $\Omega$ is  $(\delta,R_0)$-Reifenberg flat  and $[A]_{R_0}\le \delta$ for some $R_0>0$, then there exist a subsequence 
$\{u_{j'}\}_{j'}$ and a finite a.e. function $u$ with the property that $T_k(u)\in W^{1,p}_0(\Omega)$ for all $k>0$, $u_{j'}\rightarrow u$ a.e., and 
\begin{equation}\label{strong-q-w}
\nabla u_{j'} \rightarrow \nabla u \quad \text{strongly in} \quad L^q_{w}(\Omega, \mathbb{R}^n).
\end{equation}
\end{corollary}

One can also combine Theorem \ref{101120143-p} (or Theorem \ref{101120143}) with a classical result of Muckenhoupt and Wheeden \cite[Theorem 3]{MW} to obtain the following gradient regularity result. This result was shown to be sharp for fractional integrals (Riesz's potentials) of order 1 (see \cite[Theorem 4]{MW}).

\begin{corollary}\label{MW-2} Suppose that  $\frac{3n-2}{2n-1}<p\leq 2-\frac{1}{n}$. For each $f\in L^1(\Omega)$, we denote by
$u(f)$ the (unique) renormalized solution of \eqref{5hh070120148} with datum $\mu=f$ in $\Omega$. Assume that 
$1<s<n$,  $q=\frac{ns}{n-s}$, and $V(x)$ is a nonnegative function in $\RR^n$ such that 
$$K:=\sup_{Balls \, B\subset\RR^n}\left(\fint_{B} [V(x)]^q dx\right)^{\frac{1}{q}} \left( \fint_{B} [V(x)]^{\frac{-s}{s-1}}\right)^{\frac{s-1}{s}}<+\infty.$$
Then there exists   $\delta=\delta(n, p, s, \Lambda,  K)\in (0,1)$ such that if $\Omega$ is  $(\delta,R_0)$-Reifenberg flat  and $[A]_{R_0}\le \delta$ with some $R_0>0$ then  we have                            
     \begin{equation*}
\int_{\Omega} |\nabla (u(f))|^{(p-1)q}\, V(x)^{q} dx \leq C \int_{\Omega} |f(x)|^s\,  V(x)^s dx,
    \end{equation*} 
 where the constant $C$ depends only  on $n,p,s,\Lambda, K$ and $diam(\Omega)/R_0$.          
\end{corollary}

We  next describe our results in regard to equation  \eqref{Riccati}.  For this, we shall need the notion of capacity
associated to the Sobolev space $W^{1, s}(\RR^n)$,  $1<s<+\infty$.
For a compact set $K\subset\RR^n$,  we  define
\begin{equation*}
{\rm Cap}_{1,  s}(K)=\inf\Big\{\int_{\RR^n}(|\nabla \varphi|^s +\varphi^s) dx: \varphi\in C^\infty_0(\RR^n),
\varphi\geq \chi_K \Big\}.
\end{equation*}

Note that ${\rm Cap}_{1,  s}$ can be extended to all sets $E\subset\RR^n$ by letting 
\begin{equation*}
{\rm Cap}_{1,  s}(E)=\inf_{\substack{O\supset E\\ O\, open}} \, \Big\{  \sup_{\substack{K\subset O\\ K \, compact}} {\rm Cap}_{1,  s}(K)\Big\}.
\end{equation*}
Moreover, by the capacitability of Borel sets (see, e.g., \cite[Theorem 2.3.11]{AH}) we have
$${\rm Cap}_{1,  s}(E)=\sup_{\substack{K\subset E\\ K \, compact}} {\rm Cap}_{1,  s}(K)$$
for any Borel set $E\subset\RR^n$.

\begin{theorem}\label{main-Ric}
Let $\frac{3n-2}{2n-1}<p\leq 2-\frac{1}{n}$ and  $q\geq 1$. 
There exists a constant  $\delta=\delta(n, p, \Lambda, q)\in (0, 1)$ such that the following holds. 
Suppose that  $[A]_{R_0}\leq \delta$ and $\Om$ is $(\delta, R_0)$-Reifenberg flat for some $R_0>0$. Then there exists a constant 
$c_0=c_0(n, p, \Lambda, q, {\rm diam}(\Om), R_0)>0$ such that if $\mu$ is a finite signed measure in $\Om$ with 
\begin{equation}\label{capcondi} 
|\mu|(K) \leq c_0\, {\rm Cap}_{1,\, \frac{q}{q-p+1}}(K)
\end{equation}
for all compact sets $K\subset\Om$, then there exists a renormalized solution $u\in W_0^{1, q}(\Om)$ to the Riccati type equation \eqref{Riccati} such that 
\begin{equation*} 
\int_{K} |\nabla u|^q  \leq C\, {\rm Cap}_{1,\, \frac{q}{q-p+1}}(K)
\end{equation*}
for all compact sets $K\subset\Om$. Here the constant $C$ depends only on $n, p, \Lambda, q, {\rm diam}(\Om)$, and $R_0$.
\end{theorem}

It is worth mentioning that the capacitary condition \eqref{capcondi} is sharp. Namely,  if \eqref{Riccati} has a solution with $\om$ being  nonnegative and compactly supported in $\Om$ then  \eqref{capcondi} holds with  a different constant $c_0$ (see \cite{HMV, Ph1}). Moreover,  it is also practically useful. In particular, it implies  that the Marcinkiewicz space condition $\mu\in L^{\frac{n(q-p+1)}{q},  \infty}(\Om)$, $q>\frac{n(p-1)}{n-1}$, (with a small norm) is sufficient for the solvability of 
\eqref{Riccati}. Other sufficient conditions of Fefferman-Phong type involving Morrey spaces can also be deduced from  \eqref{capcondi} (see Corollaries 3.5 and 3.6 in \cite{Ph1}). See also Theorem \ref{remove-S-nec} below in which  \eqref{capcondi} is used in the study of removable singularities  for the homogeneous Riccati type equation $-{\rm div} (A(x, \nabla u)) =|\nabla u|^q$.

Theorem \ref{main-Ric} extends  similar existence  results obtained earlier for $2-\frac{1}{n}<p\leq n$ in 
\cite{55Ph2, 55Ph2-2}. See  also \cite{Ph1, Ph1-1} or \cite{AdP1, AdP2} where the case $q>p$ or $q=p$ is studied, respectively. 
In particular, Theorem \ref{main-Ric} solves  an open problem in \cite[page 13]{VHV} at least for compactly supported measures and  for  $\frac{3n-2}{2n-1}<p\leq 2-\frac{1}{n}$, $q\geq 1$. It is natural to expect that Theorem \ref{main-Ric} should also hold  for $p-1<q<1$ but we are not able to prove it here due to the lack of convexity.
It is also worth mentioning that the  `linear' case $p=2$ was first considered in the pioneering work \cite{HMV}. There is a vast literature on equations of the form \eqref{Riccati} (but mostly for $0<q\leq p$). We  refer to  \cite{ALT, BBM,   BMP1, BMP2, VHV, CC, De, FMe,  FPR,  Gre, GT,  Me, PS} and to
\cite{BMMP, FMu, FM3,  JMV1, JMV2, FV,  GMP1, GMP2,  Ph3} for various contributions.

Finally, as mentioned above  Theorem \ref{main-Ric} can be used  to give sharp bound on  the size of removable singular sets for homogeneous Riccati type equations.  We recall that a Borel set $E\subset\Omega$ is a said to be a removable singular set for the equation  $-{\rm div} (A(x, \nabla u)) =|\nabla u|^q$
in $\Omega$ if any solution $u$ to  
\begin{equation*}
\left\{\begin{array}{l}
 u\in W^{1, q}_{loc}(\Om\setminus E), ~ \text{and}\\
-{\rm div} (A(x, \nabla u)) =|\nabla u|^q ~ \text{in}~
 \mathcal{D}'(\Om\setminus E)
\end{array}
\right.
\end{equation*}
can be extended to be a solution to 
\begin{equation*}
\left\{\begin{array}{l}
 u\in W^{1, q}_{loc}(\Om), ~ \text{and}\\
-{\rm div} (A(x, \nabla u)) =|\nabla u|^q ~ \text{in}~
 \mathcal{D}'(\Om).
\end{array}
\right.
\end{equation*}

\begin{theorem}\label{remove-S-nec}
Let $\frac{3n-2}{2n-1}<p\leq 2-\frac{1}{n}$ and  $q\geq 1$.  There exists a constant  $\delta=\delta(n, p, \Lambda, q)\in (0, 1)$ such that the following holds. 
Suppose that  $[A]_{R_0}\leq \delta$ and $\Om$ is $(\delta, R_0)$-Reifenberg flat for some $R_0>0$. 
If a Borel set $E\subset\Om$ is a removable set for the equation  $-{\rm div} (A(x, \nabla u)) =|\nabla u|^q$
in $\Omega$, then it must hold that   
$${\rm Cap}_{1,\,\frac{q}{q-p+1}}(E)=0.$$ 
\end{theorem}

The proof of Theorem \ref{remove-S-nec} is based on Theorem \ref{main-Ric} and is similar to that of \cite[Theorem 3.9]{Ph1}. 

\begin{remark} By \cite[Theorem 3.8]{Ph1}, Theorem \ref{remove-S-nec} is sharp at least in the natural class of $A$-superharmonic functions in $\Omega$. Namely,
if $K$ is a compact set in $\Omega$ with ${\rm Cap}_{1,\,\frac{q}{q-p+1}}(K)=0$ then any solution $u$ to 
\begin{equation*}
\left\{\begin{array}{l}
-{\rm div} (A(x, \nabla u))\geq 0 ~in~ \mathcal{D}'(\Om),\\
 u\in W^{1, q}_{loc}(\Om\setminus K), ~ and\\
-{\rm div} (A(x, \nabla u)) =|\nabla u|^q ~ in~
 \mathcal{D}'(\Om\setminus K),
\end{array}
\right.
\end{equation*}
is also a solution to 
\begin{equation*}
\left\{\begin{array}{l}
 u\in W^{1, q}_{loc}(\Om), ~ and\\
-{\rm div} (A(x, \nabla u)) =|\nabla u|^q ~ in~
 \mathcal{D}'(\Om).
\end{array}
\right.
\end{equation*}
\end{remark}

The paper is organized as follows. In Section \ref{sec-2} we obtain some important comparison estimates that are needed for the proof of Theorem \ref{5hh23101312}.
The proof of good-$\lambda$ type bounds, Theorem \ref{5hh23101312}, is given in Section  \ref{sec-3}. Then in Section \ref{sec-4}, we prove 
Theorem \ref{101120143-p} and Corollary \ref{compactness}. Finally, we obtain existence results for the Riccati type equation \eqref{Riccati}, Theorem \ref{main-Ric},
in Section \ref{sec-5}.

    \section{Local interior  and boundary estimates}\label{sec-2}
In this section, we obtain certain local interior and boundary comparison
estimates that are essential to our development later. First let us consider
the interior ones. With $u\in W_{loc}^{1,p}(\Omega)$ and for each ball $B_{2R}=B_{2R}(x_0)\subset\subset\Omega$, we consider the unique solution $w\in W_{0}^{1,p}(B_{2R})+u$
    to the  equation 
    \begin{equation}
    \label{111120146}\left\{ \begin{array}{rcl}
    - \operatorname{div}\left( {A(x,\nabla w)} \right) &=& 0 \quad in \quad B_{2R}, \\ 
    w &=& u\quad \text{on} \quad \partial B_{2R}.  
    \end{array} \right.
    \end{equation}
		
    We first recall the following  version of interior Gehring's lemma that was proved in \cite[Theorem 6.7]{Giu}. 
    \begin{lemma} \label{111120147} Let $w$ be as in \eqref{111120146}.
    	There exist  constants $\theta_1>p$ and $C>0$ depending only on $n, p, \Lambda$ such that the following estimate      
    	\begin{equation}\label{111120148}
    	\left(\fint_{B_{\rho/2}(y)}|\nabla w|^{\theta_1} dxdt\right)^{\frac{1}{\theta_1}}\leq C\left(\fint_{B_{\rho}(y)}|\nabla w|^{p-1} dx\right)^{\frac{1}{p-1}}
    	\end{equation}holds 
    	for all  $B_{\rho}(y)\subset B_{2R}$. 
    \end{lemma}

    The next lemma gives an estimate for the difference $\nabla u-\nabla w$. This is one of the key estimates of this paper.
		We remark that earlier this kind of comparison estimates is known only in the case $p>2-\frac{1}{n}$ (see \cite{Mi2, Duzamin2}). Here we are able to obtain it for $\frac{3n-2}{2n-1}< p\leq 2-\frac{1}{n}$.
		
    \begin{lemma}\label{111120149}Let $w$ be in \eqref{111120146}. Assume that $\frac{3n-2}{2n-1}<p\leq 2-\frac{1}{n}$. Then it holds that 
    	\begin{align}\label{1111201410+}
    	\left(	\fint_{B_{2R}}|\nabla u-\nabla w|^{\gamma_0}dx\right)^{\frac{1}{\gamma_0}}&\leq C \left[\frac{|\mu|(B_{2R})}{R^{n-1}}\right]^{\frac{1}{p-1}} + \nonumber
			\\&\qquad+C\frac{|\mu|(B_{2R})}{R^{n-1}}\left(	\fint_{B_{2R}}|\nabla u|^{\gamma_0}dx\right)^{\frac{2-p}{\gamma_0}},
    	\end{align}
    	for some $\frac{2-p}{2}\leq \gamma_0<\frac{(p-1)n}{n-1}\leq 1$.  In particular, for any $\varepsilon>0$ one can find $C_\varepsilon>0$ such that  
    	\begin{align}\label{1111201410}
    \left(	\fint_{B_{2R}}|\nabla u-\nabla w|^{\gamma_0}dx\right)^{\frac{1}{\gamma_0}}\leq C_\varepsilon \left[\frac{|\mu|(B_{2R})}{R^{n-1}}\right]^{\frac{1}{p-1}}+\varepsilon \left(	\fint_{B_{2R}}|\nabla u|^{\gamma_0}dx\right)^{\frac{1}{\gamma_0}}.
    	\end{align}  
    \end{lemma}
\begin{proof} For any $\varphi\in W_0^{1,p}(B_{2R})$, we have 
\begin{align}\label{es0}
\int_{B_{2R}}\langle A(x,\nabla u)-A(x,\nabla w),\nabla\varphi\rangle dx=\int_{B_{2R}}\varphi d\mu.
\end{align}

We now set 
	\begin{equation*}
T_{h,m}(s)=\left\{ \begin{array}{l}
T_m(s)~~~~~~~~~~~~~~~~~\text{if}~~|s|\geq 2h,\\ 
	2\operatorname{sgn}(s)(|s|-h) ~~~\text{ if }~h<|s|<2h, \\ 
	0 ~~~~~~~~~~~~~~~~~~~~~~~~\text{if}~~|s|\leq h,
	\end{array} \right.
	\end{equation*}  
	for $m>2h>0$.                
 It is easy to see that we can take $\varphi=T_{h,k^{1-\alpha}}(|u-w|^{-\alpha}(u-w))$ with $\alpha\in (-\infty,1)$ and $0<h< \frac{k^{1-\alpha}}{2}$ as a test function in \eqref{es0}. This gives
 \begin{align*}
 \int_{B_{2R}\cap\{x: (2h)^{\frac{1}{1-\alpha}}<|u-w|<k\}}|u-w|^{-\alpha}g(u,w)dx\leq C k^{1-\alpha} |\mu|(B_{2R}),
 \end{align*}
 where 
 \begin{align}\label{g-def}
 g(u,w)=\frac{|\nabla (u-w)|^2}{(|\nabla w|+|\nabla u|)^{2-p}}.
 \end{align}

Thus sending $h\rightarrow 0$ we get
 \begin{align*}
 \int_{B_{2R}\cap\{x:|u-w|<k\}}|u-w|^{-\alpha}g(u,w)dx\leq C k^{1-\alpha} |\mu|(B_{2R}).
 \end{align*}

 We now estimate $|u-w|^{-\alpha}g(u,w)$ in $L^\gamma(B_{2R})$ for some appropriate $\gamma$. To do so we employ the method of \cite{bebo} (see also \cite{55Ph0}). For $k,\lambda\geq 0$, we let $$\Phi(k,\lambda)=|\{x: |u-w|>k,|u-w|^{-\alpha}g(u,w)>\lambda\}\cap B_{2R}|.$$
As $\lambda\mapsto \Phi(k,\lambda)$ is non-increasing, we find
\begin{align*}
\Phi(0,\lambda)&\leq \frac{1}{\lambda}\int_{0}^{\lambda}\Phi(0,s)ds\leq 
\Phi(k,0)+\frac{1}{\lambda}\int_{0}^{\lambda}\Phi(0,s)-\Phi(k,s) ds
\\&=|\{x:|u-w|>k\}\cap B_{2R}|+
\\&\quad  +\frac{1}{\lambda}\int_{0}^{\lambda}|\{x:|u-w|\leq k,|u-w|^{-\alpha}g(u,w)>s\}\cap B_{2R}| ds
\\&\leq k^{-\beta}\|u-w\|_{L^\beta(B_{2R})}^\beta+\frac{1}{\lambda}\int_{B_{2R}\cap\{x:|u-w|\leq k\}}|u-w|^{-\alpha}g(u,w) dx
\\&\leq k^{-\beta}\|u-w\|_{L^\beta(B_{2R})}^\beta+\frac{C k^{1-\alpha}}{\lambda} |\mu|(B_{2R}),
\end{align*} 
for any $\beta > 0$. Then choosing $$k=\left[\frac{\lambda \|u-w\|_{L^\beta(B_{2R})}^\beta}{|\mu|(B_{2R}) }\right]^{\frac{1}{1-\alpha+\beta}},$$
we obtain
\begin{align*}
\lambda^{\frac{\beta}{1-\alpha+\beta}}|\{x:|u-w|^{-\alpha}g(u,w)>\lambda\}\cap B_{2R}|\leq C |\mu|(B_{2R})^{\frac{\beta}{1-\alpha+\beta}} \|u-w\|_{L^\beta(B_{2R})}^{\frac{\beta(1-\alpha)}{1-\alpha+\beta}}
\end{align*}
for all $\lambda>0$. Thus by Holder's inequality, for $0<\gamma<\frac{\beta}{1-\alpha+\beta}$, we get 
\begin{align}\nonumber
&\int_{B_{2R}}|u-w|^{-\alpha\gamma}g(u,w)^\gamma dx
\\&\quad\leq C |B_{2R}|^{1-\frac{\gamma(1-\alpha+\beta)}{\beta}}\|[|u-w|^{-\alpha}g(u,w)]^\gamma\|_{L^{\frac{\beta}{\gamma(1-\alpha+\beta)},\infty}(B_{2R})}
\nonumber\\&\quad\leq 
C R^{n-\frac{n\gamma(1-\alpha+\beta)}{\beta}} |\mu|(B_{2R})^{\gamma}\|u-w\|_{L^\beta(B_{2R})}^{\gamma(1-\alpha)}.\label{es1}
\end{align}

We next define a quantity
$$M:=\int_{B_{2R}}|\nabla (u-w)|\, |u-w|^{-\frac{\alpha}{p}}.$$

Applying Sobolev's inequality for the function $|u-w|^{\frac{p-\alpha}{p}}$, we have 
\begin{align}
\int_{B_{2R}}|u-w|^{\frac{(p-\alpha)n}{p(n-1)}}\leq  C \left(\int_{B_{2R}}|\nabla |u-w|^{1-\frac{\alpha}{p}}|dx\right)^{\frac{n}{n-1}} =CM^{\frac{n}{n-1}}\label{es2}.
\end{align}

Then using Holder's inequality and  \eqref{es2}, we get
\begin{align}\nonumber
&\int_{B_{2R}}|\nabla (u-w)|^{\frac{(p-\alpha)n}{pn-\alpha}}\\
&\leq \left(\int_{B_{2R}}|u-w|^{\frac{(p-\alpha)n}{p(n-1)}}\right)^{\frac{\alpha(n-1)}{pn-\alpha}}\left(\int_{B_{2R}}|\nabla (u-w)| \, |u-w|^{-\frac{\alpha}{p}}dx\right)^{\frac{n(p-\alpha)}{pn-\alpha}}\nonumber\\
&\leq C M^{\frac{pn}{pn-\alpha}}\label{es9}.
\end{align}

Our next goal is to bound $M$. To this end, using $1<p<2$, we have  
\begin{align*}
|\nabla (u-w)|\leq C \left( g(u,w)^{1/p}+g(u,w)^{1/2}|\nabla u|^{\frac{2-p}{2}}\right),
\end{align*}
and thus
\begin{align}
&M \leq C\int_{B_{2R}} \left[|u-w|^{-\frac{\alpha}{p}}g(u,w)^{1/p}+|u-w|^{-\frac{\alpha}{p}}g(u,w)^{\frac{1}{2}}|\nabla u|^{\frac{2-p}{2}} \right].\label{es6}
\end{align}

We now assume that \begin{align}\label{ine1}
1/p<\frac{\beta}{1-\alpha+\beta},~~\beta=\frac{(p-\alpha)n}{p(n-1)}.
\end{align}
Thus, we can apply \eqref{es1}  to $\gamma=1/p$, to get
\begin{align}\nonumber
\int_{B_{2R}} |u-w|^{-\frac{\alpha}{p}}g(u,w)^{1/p}&\leq C R^{n-\frac{n(1-\alpha+\beta)}{p\beta}} |\mu|(B_{2R})^{1/p}\|u-w\|_{L^\beta(B_{2R})}^{(1-\alpha)/p}\\& \leq C R^{n-\frac{n(1-\alpha+\beta)}{p\beta}} |\mu|(B_{2R})^{1/p} M^{\frac{1-\alpha}{p-\alpha}}\label{es3},
\end{align}
where we used \eqref{es2} in  the last inequality.

Assume also that 
\begin{align}\label{es7}
\gamma_0:=\frac{(p-\alpha)n}{pn-\alpha}>\frac{2-p}{2}.
\end{align} 
Then by Holder's inequality with exponents $\frac{2\gamma_0}{2\gamma_0+p-2}$ and $\frac{2\gamma_0}{2-p}$,
\begin{align}\nonumber
&\int_{B_{2R}} |u-w|^{-\frac{\alpha}{p}}g(u,w)^{\frac{1}{2}}|\nabla u|^{\frac{2-p}{2}}
\\&\quad\leq \left(\int_{B_{2R}} |u-w|^{-\frac{2\alpha \gamma_0}{p(2\gamma_0+p-2)}}g(u,w)^{\frac{\gamma_0}{2\gamma_0+p-2}}\right)^{\frac{2\gamma_0+p-2}{2\gamma_0}} \left(\int_{B_{2R}} |\nabla u|^{\gamma_0}\right)^{\frac{2-p}{2\gamma_0}}\label{es5}.
\end{align}

We further restrict that 
\begin{align}\label{ine2}
\alpha<\frac{p}{2}, ~~\frac{\gamma_0}{2\gamma_0+p-2}<\frac{\beta}{1-\frac{2\alpha}{p} +\beta},
\end{align} 
which then by  \eqref{es1} gives
\begin{align}\nonumber
&\int_{B_{2R}} |u-w|^{-\frac{2\alpha \gamma_0}{p(2\gamma_0+p-2)}}g(u,w)^{\frac{\gamma_0}{2\gamma_0+p-2}}
\\&\quad \leq C R^{n-\frac{n\gamma_0(1-\frac{2\alpha}{p} +\beta)}{\beta(2\gamma_0+p-2)}} |\mu|(B_{2R})^{\frac{\gamma_0}{2\gamma_0+p-2}}\|u-w\|_{L^\beta(B_{2R})}^{\frac{\gamma_0(1-\frac{2\alpha}{p})}{2\gamma_0+p-2}}
\nonumber\\& \quad\leq C R^{n-\frac{n\gamma_0(1-\frac{2\alpha}{p} +\beta)}{\beta(2\gamma_0+p-2)}} |\mu|(B_{2R})^{\frac{\gamma_0}{2\gamma_0+p-2}}M^{\frac{\gamma_0(p-2\alpha)}{(p-\alpha)(2\gamma_0+p-2)}}.\label{es4}
\end{align}

Hence, combining \eqref{es6}, \eqref{es3}, \eqref{es5}, and \eqref{es4} we have 
\begin{align}\nonumber
M&\leq C R^{n-\frac{n(1-\alpha+\beta)}{p\beta}} |\mu|(B_{2R})^{1/p} M^{\frac{1-\alpha}{p-\alpha}}+
\\&\quad +\left(R^{n-\frac{n\gamma_0(1-\frac{2\alpha}{p} +\beta)}{\beta(2\gamma_0+p-2)}} |\mu|(B_{2R})^{\frac{\gamma_0}{2\gamma_0+p-2}}M^{\frac{\gamma_0(p-2\alpha)}{(p-\alpha)(2\gamma_0+p-2)}}\right)^{\frac{2\gamma_0+p-2}{2\gamma_0}} \times\nonumber\\
& \qquad\qquad\qquad \times\left(\int_{B_{2R}} |\nabla u|^{\gamma_0}\right)^{\frac{2-p}{2\gamma_0}}\label{es8}
\end{align}
provided  that \eqref{ine1}, \eqref{es7}, and \eqref{ine2} are satisfied.

Let $\alpha_0=\frac{\alpha}{p}<1/2$ so that $\beta=\frac{(1-\alpha_0)n}{n-1}$. We have 
\begin{align*}
1/p<\frac{\beta}{1-\alpha+\beta}~~~\Longleftrightarrow~~~p>\frac{n(2-\alpha_0)-1}{n-\alpha_0},
\end{align*}
\begin{align*}
\frac{\gamma_0}{2\gamma_0+p-2}<\frac{\beta}{1-\frac{2\alpha}{p} +\beta}~~~\Longleftrightarrow~~~p>\frac{n(2-\alpha_0)-1}{n-\alpha_0},
\end{align*}
and 
\begin{align*}
\gamma_0=\frac{(p-\alpha)n}{pn-\alpha}>\frac{2-p}{2} ~~~\Longleftrightarrow~~~p>\frac{2\alpha_0(n-1)}{n-\alpha_0}.
\end{align*}

Therefore, if  \begin{align*}
p>\frac{3n-2}{2n-1},
\end{align*} 
then  \eqref{ine1}, \eqref{es7}, and \eqref{ine2} hold  for any 
\begin{align*}
1/2>\alpha_0>\frac{-1+(2-p)n}{n-p}.
\end{align*}

With this, using Holder's inequality, we get from \eqref{es8} that
\begin{align}\nonumber
M&\leq C \left[R^{n-\frac{n(1-\alpha+\beta)}{p\beta}} |\mu|(B_{2R})^{1/p}\right]^{\frac{p-\alpha}{p-1}} +
\\&\quad+\left(R^{n-\frac{n\gamma_0(1-\frac{2\alpha}{p} +\beta)}{\beta(2\gamma_0+p-2)}} |\mu|(B_{2R})^{\frac{\gamma_0}{2\gamma_0+p-2}}\right)^{\frac{(2\gamma_0+p-2)(p-\alpha)}{p\gamma_0}} \times\nonumber
\\&\qquad\qquad \qquad \times \left(\int_{B_{2R}} |\nabla u|^{\gamma_0}\right)^{\frac{(p-\alpha)(2-p)}{p\gamma_0}}\label{es10}.
\end{align}

Thus it follows from \eqref{es9} and \eqref{es10} that 
\begin{align*}
&\left(\int_{B_{2R}}|\nabla (u-w)|^{\gamma_0}\right)^{\frac{p-\alpha}{p\gamma_0}} \leq C \left[R^{n-\frac{n(1-\alpha+\beta)}{p\beta}} |\mu|(B_{2R})^{1/p}\right]^{\frac{p-\alpha}{p-1}} +\\&\quad +\left(R^{n-\frac{n\gamma_0(1-\frac{2\alpha}{p} +\beta)}{\beta(2\gamma_0+p-2)}} |\mu|(B_{2R})^{\frac{\gamma_0}{2\gamma_0+p-2}}\right)^{\frac{(2\gamma_0+p-2)(p-\alpha)}{p\gamma_0}} \left(\int_{B_{2R}} |\nabla u|^{\gamma_0}\right)^{\frac{(p-\alpha)(2-p)}{p\gamma_0}}.
\end{align*}

That is, we obtain \eqref{1111201410+} with $\frac{2-p}{2}< \gamma_0<\frac{(p-1)n}{n-1}\leq 1$ as desired.

Finally, using Young's inequality, we get the bound \eqref{1111201410} which completes the proof of the lemma. 
\end{proof}

\medskip
The following proposition provides a useful estimate for the difference $\nabla u-\nabla v$ for a well-controlled locally Lipschitz function $v$.

\begin{proposition} \label{inter} Let $\mu\in\mathfrak{M}_b(\Omega)$ and $\frac{3n-2}{2n-1}<p\leq 2-\frac{1}{n}$. Let $\gamma_0$ be as in Lemma \ref{111120149}. There exists $v\in W^{1,p}(B_R)\cap W^{1,\infty}(B_{R/2})$ such that for any $\varepsilon>0$, 
	\begin{align*}
	\|\nabla v\|_{L^\infty(B_{R/2})}\leq C \left[\frac{|\mu|(B_{2R})}{R^{n-1}}\right]^{\frac{1}{p-1}}+C \left(\fint_{B_{2R}}|\nabla u|^{\gamma_0}\right)^{1/\gamma_0},
	\end{align*}
	and
	\begin{align*}
	\left(\fint_{B_{R}}|\nabla u-\nabla v|^{\gamma_0}dx\right)^{\frac{1}{\gamma_0}}&\leq C_\varepsilon \left[\frac{|\mu|(B_{2R})}{R^{n-1}}\right]^{\frac{1}{p-1}}+\\
	& \qquad \qquad C(([A]_{R_0})^{\kappa} +\varepsilon)\left(\fint_{B_{2R}}|\nabla u|^{\gamma_0}\right)^{1/\gamma_0}.
	\end{align*}
for some $C_\varepsilon=C(n,p,\Lambda,\varepsilon)>0$. Here $\kappa$ is a constant in $(0,1)$.
\end{proposition}
\begin{proof} By \cite[Lemma 2.3 and Corollary 2.4]{55Ph2}, there exists $v\in W^{1,p}(B_R)\cap W^{1,\infty}(B_{R/2})$ such that 
	\begin{align*}
	\|\nabla v\|_{L^\infty(B_{R/2})}\leq C \left(\fint_{B_R}|\nabla w|^p\right)^{1/p},
	\end{align*}
	and 
	\begin{align*}
	\fint_{B_{R}}|\nabla w-\nabla v|dx\leq C ([A]_{R_0})^{\kappa} \left(\fint_{B_R}|\nabla w|^p\right)^{1/p},
	\end{align*}
	 for some $\kappa\in (0,1).$ Combining these with \eqref{111120148} in Lemma \ref{111120147},  \eqref{1111201410} in Lemma \ref{111120149}, we get the desired results.  
\end{proof}

\medskip

      Next, we focus on the corresponding estimates near the boundary. We recall that $\Omega$ is $(\delta_0,R_0)$-Reifenberg flat  with $\delta_0<1/2$. Fix $x_0\in \Omega$ and $0<R<R_0/10$. With $u\in W_0^{1,p}(\Omega)$ being a solution to \eqref{5hh070120148}, 
      we now consider the unique solution $w\in W_{0}^{1,p}(\Omega_{10R}(x_0))+u$
       to the following equation 
      \begin{equation}
      \label{111120146*}\left\{ \begin{array}{rcl}
      - \operatorname{div}\left( {A(x,\nabla w)} \right) &=& 0 \quad ~~~\text{in}\quad \Omega_{10R}(x_0), \\ 
      w &=& u\quad \quad \text{on} \quad \partial \Omega_{10R}(x_0). 
      \end{array} \right.
      \end{equation}
      Hereafter, the notation $\Omega_r(x)$ indicates the set $\Omega\cap B_r(x)$. 
      By \cite[Lemma 2.5]{55Ph0}, we have the following boundary counterpart of Lemma \ref{111120147}.
      \begin{lemma} \label{111120147*} Let $w$ be as in \eqref{111120146*}.
      	There exist  constants $\theta_1>p$ and $C>0$ depending only on $n,p,\delta_0,\Lambda$ such that the following estimate      
      	\begin{equation}\label{111120148*}
      	\left(\fint_{B_{\rho/2}(y)}|\nabla w|^{\theta_1} dxdt\right)^{\frac{1}{\theta_1}}\leq C\left(\fint_{B_{3\rho}(y)}|\nabla w|^{p-1} dx\right)^{\frac{1}{p-1}},
      	\end{equation}holds 
      	for all  $B_{3\rho}(y)\subset B_{10R}(x_0)$. 
      \end{lemma}

  We also have the following analogues of Lemmas  \ref{111120149}.
  \begin{lemma}\label{111120149"} Assume that $\frac{3n-2}{2n-1}<p\leq 2-\frac{1}{n}$. Let $w$ be as in \eqref{111120146*} and $\gamma_0$ be as in Lemma  \ref{111120149}. Then we  have
  	\begin{align*}\nonumber
  \left(	\fint_{B_{10R}(x_0)}|\nabla (u-w)|^{\gamma_0}dx\right)^{\frac{1}{\gamma_0}}&\leq C\left[\frac{|\mu|(B_{10R}(x_0))}{R^{n-1}}\right]^{\frac{1}{p-1}}+\\
&\qquad +	C \frac{|\mu|(B_{10R}(x_0))}{R^{n-1}}\left(	\fint_{B_{10R}(x_0)}|\nabla u|^{\gamma_0}dx\right)^{\frac{2-p}{\gamma_0}}.
  	\end{align*}
   In particular, for any $\varepsilon>0$,
  	 	\begin{align}\label{1111201410"}
  	 \left(	\fint_{B_{10R}(x_0)}|\nabla (u- w)|^{\gamma_0}dx\right)^{\frac{1}{\gamma_0}}&\leq C_\varepsilon \left[\frac{|\mu|(B_{10R}(x_0))}{R^{n-1}}\right]^{\frac{1}{p-1}}+\nonumber\\
		&\qquad \varepsilon \left(	\fint_{B_{10R}(x_0)}|\nabla u|^{\gamma_0}dx\right)^{\frac{1}{\gamma_0}}.
  	 \end{align}
  \end{lemma}

Using Lemma \ref{111120149"} we  derive the following boundary version of Proposition \ref{inter}.
   \begin{proposition} \label{boundary} Let $\mu\in\mathfrak{M}_b(\Omega)$ and $\frac{3n-2}{2n-1}<p\leq 2-\frac{1}{n}$.  Let $\gamma_0$ be as in Lemma  \ref{111120149}. For any $\varepsilon>0$, there exists $\delta_0=\delta_0(n,p,\Lambda,\varepsilon)\in(0,1)$ such that the following holds. If $\Omega$ is $(\delta_0,R_0)$-Reifenberg flat  and $u\in W_{0}^{1,p}(\Omega)$, $x_0\in\partial \Omega$, and $0<R<R_0/10$, then there exists a function $V\in W^{1,\infty}(B_{R/10}(x_0))$ 
such that   	
\begin{align*}
   	\|\nabla V\|_{L^\infty(B_{R/10}(x_0))}\leq C \left[\frac{|\mu|(B_{10R}(x_0))}{R^{n-1}}\right]^{\frac{1}{p-1}}+C \left(\fint_{B_{10R}}|\nabla u|^{\gamma_0}\right)^{1/\gamma_0},
   	\end{align*}
   	and
   	\begin{align*}
   &\left(\fint_{B_{R/10}(x_0)}|\nabla (u-V)|^{\gamma_0}dx\right)^{\frac{1}{\gamma_0}}\\&~~~~~\leq C_\varepsilon \left[\frac{|\mu|(B_{10R}(x_0))}{R^{n-1}}\right]^{\frac{1}{p-1}}+C(([A]_{R_0})^{\kappa} +\varepsilon)\left(\fint_{B_{10R}(x_0)}|\nabla u|^{\gamma_0}\right)^{1/\gamma_0}.
   	\end{align*}
   	for some $C_\varepsilon=C(n,p,\Lambda,\varepsilon)>0$. Here $\kappa$ is a constant in $(0,1)$.
   \end{proposition} 
\begin{proof}
By \cite[Corollary 2.13]{55Ph2}, for any $\varepsilon>0$, there exists $\delta_0=\delta_0(n,p,\Lambda,\varepsilon)\in(0,1)$ such that if $\Omega$ is a $(\delta_0,R_0)$-Reifenberg flat domain  then we can find  $V\in W^{1,\infty}(B_{R/10}(x_0))$ satisfying  
\begin{align*}
\|\nabla V\|_{L^\infty(B_{R/10}(x_0))}\leq C \left(\fint_{B_R(x_0)}|\nabla w|^p\right)^{1/p},
\end{align*}
and 
\begin{align*}
\fint_{B_{R/10}(x_0)}|\nabla w-\nabla V|dx\leq C (([A]_{R_0})^{\kappa}+\varepsilon) \left(\fint_{B_R(x_0)}|\nabla w|^p\right)^{1/p},
\end{align*}
for some $\kappa\in (0,1)$. Combining these with \eqref{111120148*} in Lemma \ref{111120147*},  \eqref{1111201410"} in Lemma \ref{111120149"}, we arrive at the conclusion.  
\end{proof}

\section{Good-$\lambda$ type bounds on Reifenberg flat domains}\label{sec-3}

The purpose of this section is to prove Theorem  \ref{5hh23101312}.
 Our main tools here are Propositions \ref{inter} and \ref{boundary} and Lemma \ref{5hhvitali2} below.
This lemma can be viewed as a substitution for the Calder\'on-Zygmund-Krylov-Safonov decomposition.
The weighted version that is used here was obtained in \cite{55MePh2}. See also \cite{Wa, BW1, CaPe} for the case 
the weight $w\equiv 1$.

\begin{lemma}\label{5hhvitali2} Let $\Omega$ be a $(\delta,R_0)$-Reifenberg flat domain with $\delta<1/4$ and let $w$ be an $\mathbf{A}_\infty$ weight. Suppose that the sequence of balls $\{B_r(y_i)\}_{i=1}^L$ with centers $y_i\in\overline{\Omega}$ and  radius $r\leq R_0/4$ covers $\Omega$.  Let $E\subset F\subset \Omega$ be measurable sets for which 
 there exists $0<\varepsilon<1$ such that  
\begin{description}
 \item[1.] $w(E)<\varepsilon w(B_r(y_i))$ for all $i=1,...,L$, and 
 \item[2.] for all $x\in \Omega$, $\rho\in (0,2r]$, we have       
	 $w(E\cap B_\rho(x))\geq \varepsilon w(B_\rho(x)) \Longrightarrow B_\rho(x)\cap \Omega\subset F$. 
	\end{description}
 Then $
	w(E)\leq C \varepsilon w(F)$         
	for a constant $C$ depending only on $n$ and $[w]_{\mathbf{A}_\infty}$.
\end{lemma}

We are now ready to prove Theorem \ref{5hh23101312}.

\medskip

\noindent \begin{proof}[Proof of Theorem \ref{5hh23101312}]  We shall use some of the ideas in the proofs of \cite[Theorem 1.4]{55Ph2}
and  \cite[Theorem 8.4]{55QH2} (see also \cite[Theorem 3.1]{55QH3}).

Let $\gamma_0$ be as in Lemma  \ref{111120149}  and let $u$ be a renormalized solution of \eqref{5hh070120148}. We first recall from \cite[Theorem 4.1]{11DMOP} that 
	\begin{align*}
	\|\nabla u\|_{L^{\frac{(p-1)n}{n-1},\infty}(\Omega)}\leq C\left[|\mu|(\Omega)\right]^{\frac{1}{p-1}},
	\end{align*}
	which implies  that
	\begin{align}\label{es14}
\left(	\frac{1}{R^n}\int_{\Omega}|\nabla u|^{\gamma}\right)^{1/\gamma}\leq C_\gamma \left[\frac{|\mu|(\Omega)}{R^{n-1}}\right]^{\frac{1}{p-1}}, \qquad \text{with}\quad R=diam(\Omega),
	\end{align}
	for any $\gamma\in \left(0,\frac{(p-1)n}{n-1}\right)$.
	
For $k>0$, let $\mu_0,\lambda_k^+,\lambda_k^-$ be as in Definition \ref{derenormalized}. Let $u_k\in W_0^{1,p}(\Omega)$ be the unique solution of the equation
	\begin{equation*}
	\left\{
	\begin{array}[c]{rcl}
	-\text{div}(A(x,\nabla u_k))&=&\mu_{k} \quad \text{in } \quad\Omega,\\
	{u}_{k}&=&0\quad\text{on } \quad \partial\Omega,\\
	\end{array}
	\right.  
	\end{equation*}
	where we set $\mu_k=\chi_{\{|u|<k\}}\mu_0+\lambda_k^+-\lambda_k^-$. 
	 
Note that  we have  $u_k=T_k(u)$ and $\mu_k\rightarrow \mu$ 
in the narrow topology of measures (see \cite[Remark 2.32]{11DMOP}). Thus,  
	\begin{equation}\label{grad-appr}
	\nabla u_k\to \nabla u  \quad \text{in} \quad L^\gamma(\Omega) \quad\forall \gamma\in \left(0,\frac{(p-1)n}{n-1}\right).
	\end{equation}
	
Let us set $$E_{\lambda,\delta_2}=\{({\bf M}(|\nabla u|^{\gamma_0}))^{1/\gamma_0}>\Lambda_0\lambda, (\mathbf{M}_1(\mu))^{\frac{1}{p-1}}\le \delta_2\lambda \}\cap \Omega,$$ and $$F_\lambda=\{ ({\bf M}(|\nabla u|^{\gamma_0}))^{1/\gamma_0}> \lambda\}\cap \Omega,$$ for $\delta_2\in (0,1)$ and $\lambda>0$. Here $\Lambda_0$ is a constant depending only on $n,p,\gamma_0,\Lambda$ and is to be  chosen  later.  
	Also, let $\{y_i\}_{i=1}^L\subset \Omega$ and a ball $B_0$ with radius $2R$ such that 
	 $$
	 \Omega\subset \bigcup\limits_{i = 1}^L {{B_{r_0}}({y_i})}  \subset {B_0},$$
	 where $r_0=\min\{R_0/1000,R\}$.

	We now claim that
	 \begin{equation}\label{5hh2310131}
	 w(E_{\lambda,\delta_2})\leq \varepsilon w({B_{r_0}}({y_i})) ~~\forall \lambda>0, \forall i=1,2,\dots,L,
	 \end{equation}
	 provided $\delta_2=\delta_2(n,p,\Lambda,\epsilon,[w]_{\mathbf{A}_\infty},R/R_0)>0$ is small enough.

	 Indeed, we may assume that $E_{\lambda,\delta_2}\not=\emptyset$ and thus 
	$$|\mu| (\Omega)\leq R^{n-1}(\delta_2\lambda)^{p-1}.$$
	
	Since  ${\bf M}$ is a bounded operator from $L^1(\mathbb{R}^{n})$ into $L^{1,\infty}(\mathbb{R}^{n})$, in view of \eqref{es14} with $\gamma=\gamma_0$ we find
	 \begin{align*}
	  |E_{\lambda,\delta_2}|&\leq \frac{C}{(\Lambda_0\lambda)^{\gamma_0}}\int_{\Omega}|\nabla u|^{\gamma_0}dx \leq \frac{CR^n}{(\Lambda_0\lambda)^{\gamma_0}} \left[\frac{|\mu|(\Omega)}{R^{n-1}}\right]^{\frac{\gamma_0}{p-1}}.
	 \end{align*}

	 Thus we obtain
\begin{align*}
 |E_{\lambda,\delta_2}|\leq \frac{CR^n}{(\Lambda_0\lambda)^{\gamma_0}} \left[\frac{R^{n-1}(\delta_2\lambda)^{p-1}}{R^{n-1}}\right]^{\frac{\gamma_0}{p-1}}=C\delta_2^{\gamma_0}|B_0|.
\end{align*}
  
Hence using the property of  ${\bf A}_\infty$ weights we have 
	 \begin{align*}
	 	w(E_{\lambda,\delta_2})\leq c\left(\frac{|E_{\lambda,\delta_2}|}{|B_{0}|}\right)^\nu w(B_0)\leq C\delta_2^{\nu\gamma_0} w(B_0),
	 \end{align*}
	 where $(c,\nu)$ is a pair of ${\bf A}_\infty$ constants of $w$. It is known that (see, e.g., \cite{55Gra}) there exist $c_1=c_1(n,c,\nu)$ and $\nu_1=\nu_1(n,c,\nu)$ such that 
	 \begin{equation*}
	 \frac{w(B_0)}{w({B_{r_0}}({y_i}))}\leq c_1\left(\frac{|B_0|}{|{B_{r_0}}({y_i})|}\right)^{\nu_1}~~\forall i=1,2,\dots,L.
	 \end{equation*}

	 Thus we obtain 
	 \begin{align*}
	 w(E_{\lambda,\delta_2})\leq C\delta_2^{\nu\gamma_0} \left(\frac{|B_0|}{|{B_{r_0}}({y_i})|}\right)^{\nu_1} w({B_{r_0}}({y_i}))
	 < \varepsilon w({B_{r_0}}({y_i}))~~\forall i=1,2,\dots,L,
	 \end{align*}
	provided  $\delta_2$ is small enough depending on $n,p,\gamma_0,\epsilon,[w]_{\mathbf{A}_\infty},R/R_0$. This proves  \eqref{5hh2310131}.

	 Next we verify that for all $x\in \Omega$, $r\in (0,2r_0]$, and $\lambda>0$ we have
	 \begin{equation}\label{2nd-check}
	 w(E_{\lambda,\delta_2}\cap B_r(x))\geq \varepsilon w(B_r(x)) \Longrightarrow B_r(x)\cap \Omega\subset F_\lambda,
	 \end{equation}
	 provided $\delta_2$ is small enough depending on $n,p,\Lambda, \gamma_0,\epsilon,[w]_{\mathbf{A}_\infty},R/R_0$ .

	Indeed,  take $x\in \Omega$ and $0<r\leq 2r_0$.
	 By contraposition, assume that $B_r(x)\cap \Omega\cap F^c_\lambda\not= \emptyset$ and $E_{\lambda,\delta_2}\cap B_r(x)\not = \emptyset$ i.e., there exist $x_1,x_2\in B_r(x)\cap \Omega$ such that $\left[{\bf M}(|\nabla u|^{\gamma_0})(x_1)\right]^{1/\gamma_0}\leq \lambda$ and $\mathbf{M}_1(\mu)(x_2)\le (\delta_2 \lambda)^{p-1}$.
	 We need to prove that
	 \begin{equation}\label{5hh2310133}
	 w(E_{\lambda,\delta_2}\cap B_r(x))< \varepsilon w(B_r(x)). 
	 \end{equation}

	 Clearly,
	 \begin{equation*}
	 {\bf M}(|\nabla u|)(y)\leq \max\{\left[{\bf M}\left(\chi_{B_{2r}(x)}|\nabla u|^{\gamma_0}\right)(y)\right]^{\frac{1}{\gamma_0}},3^{n}\lambda\}~~\forall y\in B_r(x).
	 \end{equation*}

	 Therefore, for all $\lambda>0$ and $\Lambda_0\geq 3^{n}$,
	 \begin{eqnarray*}\label{5hh2310134}E_{\lambda,\delta_2}\cap B_r(x)=\{{\bf M}\left(\chi_{B_{2r}(x)}|\nabla u|^{\gamma_0}\right)^{\frac{1}{\gamma_0}}>\Lambda_0\lambda, (\mathbf{M}_{1}(\mu))^{\frac{1}{p-1}}\leq \delta_2\lambda\}\cap \Omega \cap B_r(x).
	 \end{eqnarray*}

	 To prove \eqref{5hh2310133} we separately consider  the case $B_{8r}(x)\subset\subset\Omega$ and the case $B_{8r}(x)\cap\Omega^{c}\not=\emptyset$.
	
	\medskip
	
	 \noindent {\bf 1. The case $B_{8r}(x)\subset\subset\Omega$:} Applying  Proposition \ref{inter} to  $u=u_{k}\in W_{0}^{1,p}(\Omega),\mu=\mu_k$ and $B_{2R}=B_{8r}(x)$, there is a function $v_k\in W^{1,p}(B_{4r}(x))\cap W^{1,\infty}(B_{2r}(x))$ such that for any $\eta>0$, 
	 \begin{align*}
	 \|\nabla v_k\|_{L^\infty(B_{2r}(x))}\leq C \left[\frac{|\mu_k|(B_{8r}(x))}{r^{n-1}}\right]^{\frac{1}{p-1}}+C \left(\fint_{B_{8r}(x)}|\nabla u_k|^{\gamma_0}\right)^{1/\gamma_0},
	 \end{align*}
	 and
	 \begin{align*}
	 &\left(\fint_{B_{4r}}|\nabla u_k-\nabla v_k|^{\gamma_0}dx\right)^{\frac{1}{\gamma_0}}\nonumber\\
        &\qquad \qquad\leq C_\eta \left[\frac{|\mu_k|(B_{8r}(x))}{r^{n-1}}\right]^{\frac{1}{p-1}}+C(([A]_{R_0})^{\kappa} +\eta)\left(\fint_{B_{8r}}|\nabla u_k|^{\gamma_0}\right)^{1/\gamma_0},
	 \end{align*}
	 for some $\kappa\in (0,1)$.
	
	 Using $\left[{\bf M}(|\nabla u|^{\gamma_0})(x_1)\right]^{1/\gamma_0}\leq \lambda$ and  $[\mathbf{M}_1(\mu)(x_2)]^{\frac{1}{p-1}}\le \delta_2 \lambda$  with 
        $x_1,x_2\in B_r(x)$, and property \eqref{grad-appr},
			 we get 
	 \begin{align*}
	 \mathop {\limsup }\limits_{k \to \infty } \|\nabla v_{k}\|_{L^\infty(B_{2r}(x))}&\leq   C \left[\frac{|\mu|(\overline{B_{8r}(x)})}{r^{n-1}}\right]^{\frac{1}{p-1}}+C \left(\fint_{B_{8r}(x)}|\nabla u|^{\gamma_0}\right)^{1/\gamma_0}    \\&\leq   C [\mathbf{M}_1(\mu)(x_2)]^{\frac{1}{p-1}}+C \left[{\bf M}(|\nabla u|^{\gamma_0})(x_1)\right]^{1/\gamma_0}                     
	 \\&\leq C\lambda,  
	 \end{align*}
	 and 
	 \begin{align*}
	 &\mathop {\limsup }\limits_{k \to \infty }  \left(\fint_{B_{4r}(x)}|\nabla u_k-\nabla v_k|^{\gamma_0}dx\right)^{\frac{1}{\gamma_0}}\\
          &\qquad\leq C_\eta \left[\frac{|\mu|(\overline{B_{8r}(x)})}{r^{n-1}}\right]^{\frac{1}{p-1}}+C(([A]_{R_0})^{\kappa} +\eta)      \left(\fint_{B_{8r}(x)}|\nabla u|^{\gamma_0}\right)^{1/\gamma_0} \\
	 &\qquad\leq C_\eta [\mathbf{M}_1(\mu)(x_2)]^{\frac{1}{p-1}}+C(([A]_{R_0})^{\kappa} +\eta) \left[{\bf M}(|\nabla u|^{\gamma_0})(x_1)
         \right]^{1/\gamma_0}            
	 \\&\qquad\leq C\left(C_{\eta}\delta_2+\delta_1^{\kappa}+\eta\right)\lambda.
	 \end{align*}                                                                
	Here also we used that $\mu_k\rightarrow \mu$ in the narrow topology of measures
	and that 	$[A]_{R_0}\leq \delta_1$.

	 Thus there exists $k_0>1$ such that for all $k\geq k_0$ we have  
	 \begin{equation}\label{5hh2310136}
	 \|\nabla v_{k}\|_{L^\infty(B_{2r}(x))}\leq C\lambda,
	 \end{equation}                        
and
	 \begin{equation}\label{5hh2310137} 
	 \left(\fint_{B_{4r}(x)}|\nabla u_k-\nabla v_k|^{\gamma_0}dx\right)^{\frac{1}{\gamma_0}}\leq C\left(C_{\eta}\delta_2+\delta_1^{\kappa}+\eta\right)\lambda.
	 \end{equation}

	  Since 
	  \begin{align*}
	  (	{\bf M}(|\sum_{j=1}^{3}f_j|^{\gamma_0}))^{1/\gamma_0}\leq 3\sum_{j=1}^{3}	({\bf M}(|f_j|^{\gamma_0}))^{1/\gamma_0},
	  \end{align*}       
	   we find
	  \begin{align}\nonumber
	  |E_{\lambda,\delta_2}\cap B_r(x)|&\leq   |\{{\bf M}\left(\chi_{B_{2r}(x)}|\nabla (u_k-v_k)|^{\gamma_0}\right)^{\frac{1}{\gamma_0}}>\Lambda_0\lambda/9\}\cap B_r(x)|
	  \\&\nonumber+ |\{{\bf M}\left(\chi_{B_{2r}(x)}|\nabla (u-u_k)|^{\gamma_0}\right)^{\frac{1}{\gamma_0}}>\Lambda_0\lambda/9\}\cap B_r(x)|\\&+
	  |\{{\bf M}\left(\chi_{B_{2r}(x)}|\nabla v_k|^{\gamma_0}\right)^{\frac{1}{\gamma_0}}>\Lambda_0\lambda/9\}\cap B_r(x)|.         \label{es18}                 
	  \end{align}  

	 In view of \eqref{5hh2310136} we see that for $\Lambda_0\geq \max\{3^{n},10C\}$ ($C$ is the constant in \eqref{5hh2310136}) and $k\geq k_0$, it holds that 
	 \begin{align*}
	 |\{{\bf M}\left(\chi_{B_{2r}(x)}|\nabla v_k|^{\gamma_0}\right)^{\frac{1}{\gamma_0}}>\Lambda_0\lambda/9\}\cap B_r(x)|=0.
	 \end{align*}

Thus, we deduce from \eqref{es18} and \eqref{5hh2310137} that for $k\geq k_0$,
	 \begin{align*}
	 |E_{\lambda,\delta_2}\cap B_r(x)|&\leq   |\{\left[{\bf M}\left(\chi_{B_{2r}(x)}|\nabla (u_k-v_k)|^{\gamma_0}\right)\right]^{\frac{1}{\gamma_0}}>\Lambda_0\lambda/9\}\cap B_r(x)|
	 \\&+ |\{\left[{\bf M}\left(\chi_{B_{2r}(x)}|\nabla (u-u_k)|^{\gamma_0}\right)\right]^{\frac{1}{\gamma_0}}>\Lambda_0\lambda/9\}\cap B_r(x)|\\&\leq \frac{C}{\lambda^{\gamma_0}}    \left[\int_{B_{2r}(x)}  |\nabla (u_k-v_k)|^{\gamma_0}+ \int_{B_{2r}(x)}  |\nabla (u-u_k)|^{\gamma_0} \right]    
	 \\&\leq \frac{C}{\lambda^{\gamma_0}}    \left[\left(C_{\eta}\delta_2+\delta_1^{\kappa}+\eta\right)^{\gamma_0}\lambda^{\gamma_0}r^n+ \int_{B_{2r}(x)}  |\nabla (u-u_k)|^{\gamma_0} \right].
	 \end{align*}

	 At this point, letting $k\to\infty$  we get 
	 \begin{equation*}
	 |E_{\lambda,\delta_2}\cap B_r(x)|\leq C \left(C_{\eta}\delta_2+\delta_1^{\kappa}+\eta\right)^{\gamma_0}|B_r(x)|.
	 \end{equation*}

	 Thus,  
	 \begin{align*}
	 w(E_{\lambda,\delta_2}\cap B_r(x))&\leq c\left(\frac{|E_{\lambda,\delta_2}\cap B_r(x) |}{|B_r(x)|}\right)^\nu w(B_r(x))
	 \\&\leq  c\left(C_{\eta}\delta_2+\delta_1^{\kappa}+\eta\right)^{\gamma_0\nu} w(B_r(x))
	 \\&< \varepsilon w(B_r(x)),
	 \end{align*} 
	 where $\eta,\delta_1\leq C(n,p, \Lambda,\gamma_0,\epsilon,[w]_{\mathbf{A}_\infty})$ and
	$\delta_2\leq C(n,p, \Lambda,\gamma_0,\epsilon,[w]_{\mathbf{A}_\infty}, R/R_0)$.
	
	\medskip
	
	 \noindent {\bf 2. The case $B_{8r}(x)\cap\Omega^{c}\not=\emptyset$:} Let $x_3\in\partial \Omega$ such that $|x_3-x|=\text{dist}(x,\partial\Omega)$.  We have 
	 \begin{equation}\label{5hh2310138}
	 B_{2r}(x)\subset B_{10r}(x_3)\subset B_{100r}(x_3)\subset B_{108r}(x)\subset B_{109r}(x_1),
	 \end{equation}
	 and 
	 \begin{equation}\label{5hh2310139}
	 B_{100r}(x_3)\subset B_{108r}(x)\subset B_{109r}(x_2).
	 \end{equation}

	  Applying  Proposition \ref{boundary} to  $u=u_{k}\in W_{0}^{1,p}(\Omega),\mu=\mu_k$ and $B_{10R}=B_{100r}(x_3)$, 	for any $\eta>0$ there exists $\delta_0=\delta_0(n,p,\Lambda,\eta)$ such that the following holds. If $\Omega$ is a $(\delta_0,R_0)$-Reifenberg flat domain, there exists a function $V_k\in W^{1,\infty}(B_{10r}(x_3))$  such that 
	  \begin{align*}
	  \|\nabla V_k\|_{L^\infty(B_{10r}(x_3))}\leq C \left[\frac{|\mu_k|(B_{100r}(x_3))}{r^{n-1}}\right]^{\frac{1}{p-1}}+C \left(\fint_{B_{100r}(x_3)}|\nabla u_k|^{\gamma_0}\right)^{1/\gamma_0},
	  \end{align*}
and		
	  \begin{align*}\nonumber &
	  \left(\fint_{B_{10r}(x_3)}|\nabla (u_k-V_k)|^{\gamma_0}dx\right)^{\frac{1}{\gamma_0}}\\&\qquad\leq C_\eta \left[\frac{|\mu_k|(B_{100r}(x_3))}{r^{n-1}}\right]^{\frac{1}{p-1}}+C(([A]_{R_0})^{\kappa} +\eta)\left(\fint_{B_{100r}(x_3)}|\nabla u_k|^{\gamma_0}\right)^{1/\gamma_0},
	  \end{align*}
	  for some  $\kappa\in (0,1)$.

	 Since $\left[{\bf M}(|\nabla u|^{\gamma_0})(x_1)\right]^{1/\gamma_0}\leq \lambda$ and  $[\mathbf{M}_1(\mu)(x_2)]^{\frac{1}{p-1}}\le \delta_2 \lambda$  with $x_1,x_2\in B_r(x)$, by \eqref{5hh2310138}, \eqref{5hh2310139}, the fact that $[A]_{R_0}\leq \delta_1$,  and property \eqref{grad-appr},
 we get 
	 \begin{align*}
	 \mathop {\limsup }\limits_{k \to \infty } \|\nabla V_k\|_{L^\infty(B_{2r}(x))}&\leq C \left[\frac{|\mu|(\overline{B_{100r}(x_3)})}{r^{n-1}}\right]^{\frac{1}{p-1}}+C \left(\fint_{B_{100r}(x_3)}|\nabla u|^{\gamma_0}\right)^{1/\gamma_0}
	 \\& \leq C \left[\frac{|\mu|(B_{109r}(x_2))}{r^{n-1}}\right]^{\frac{1}{p-1}}+C \left(\fint_{B_{109r}(x_1)}|\nabla u|^{\gamma_0}\right)^{1/\gamma_0}
	 \\&\leq C\left([\mathbf{M}_1(\mu)(x_2)]^{\frac{1}{p-1}}+\left[{\bf M}(|\nabla u|^{\gamma_0})(x_1)\right]^{1/\gamma_0}\right)
	 \\&\leq C\lambda,
	 \end{align*}
	 and 
	 \begin{align*}
	 &\mathop {\limsup }\limits_{k \to \infty }\left(\fint_{B_{2r}(x)}|\nabla (u_k-V_k)|^{\gamma_0}dx\right)^{\frac{1}{\gamma_0}}\\&~~~~~~~~\leq C_\eta [\mathbf{M}_1(\mu)(x_2)]^{\frac{1}{p-1}}+C(([A]_{R_0})^{\kappa} +\eta)\left[{\bf M}(|\nabla u|^{\gamma_0})(x_1)\right]^{1/\gamma_0}
	 \\&~~~~~~~~\leq C\left( C_\eta \delta_2+\delta_1^{\kappa}+\eta\right)\lambda.
	 \end{align*}

	Thus we can find $k_0>1$ such that for all $k\geq k_0$ we have   
	 \begin{align}\label{5hh23101310}
	 \|\nabla V_k\|_{L^\infty(B_{2r}(x))}\leq C\lambda,
	 \end{align}
	 and 
	 \begin{align}\label{5hh23101311}
	 \left(\fint_{B_{2r}(x)}|\nabla (u_k-V_k)|^{\gamma_0}dx\right)^{\frac{1}{\gamma_0}}\leq C\left( C_\eta \delta_2+\delta_1^{\kappa}+\eta\right)\lambda.
	 \end{align}

	  As in the interior case we also have  for $k\geq k_0$,
			\begin{align*}
	  |E_{\lambda,\delta_2}\cap B_r(x)|&\leq   |\{{\bf M}\left(\chi_{B_{2r}(x)}|\nabla (u_k-v_k)|^{\gamma_0}\right)^{\frac{1}{\gamma_0}}>\Lambda_0\lambda/9\}\cap B_r(x)|
	  \\&\qquad+ |\{{\bf M}\left(\chi_{B_{2r}(x)}|\nabla (u-u_k)|^{\gamma_0}\right)^{\frac{1}{\gamma_0}}>\Lambda_0\lambda/9\}\cap B_r(x)|,
	  \end{align*}                      
	  for a constant $\Lambda_0>1$ depending only on $n,p,\Lambda$. 
	 Therefore, we deduce from \eqref{5hh23101310} and \eqref{5hh23101311} that, for  $k\geq k_0$,
	 \begin{align*}
	 |E_{\lambda,\delta_2}\cap B_r(x)|&\leq  \frac{C}{\lambda^{\gamma_0}}    \left(\int_{B_{2r}(x)}  |\nabla (u_k-v_k)|^{\gamma_0}+ \int_{B_{2r}(x)}  |\nabla (u-u_k)|^{\gamma_0} \right)    
	 \\&\leq \frac{C}{\lambda^{\gamma_0}}    \left(\left(C_{\eta}\delta_2+\delta_1^{\kappa}+\eta\right)^{\gamma_0}\lambda^{\gamma_0}r^n+ \int_{B_{2r}(x)}  |\nabla (u-u_k)|^{\gamma_0} \right).    
	 \end{align*}

Then 	 letting $k\to\infty$  we get 
	\begin{equation*}
	|E_{\lambda,\delta_2}\cap B_r(x)|\leq C \left(C_{\eta}\delta_2+\delta_1^{\kappa}+\eta\right)^{\gamma_0}|B_r(x)|.
	\end{equation*}

	 Thus we have 
        \begin{align*}
	 w(E_{\lambda,\delta_2}\cap B_r(x))&\leq c\left(\frac{|E_{\lambda,\delta_2}\cap B_r(x) |}{|B_r(x)|}\right)^\nu w(B_r(x))
	 \\&\leq  c\left(C_{\eta}\delta_2+\delta_1^{\kappa}+\eta\right)^{\gamma_0\nu} w(B_r(x))
	 \\&< \varepsilon w(B_r(x)).
	 \end{align*} 
	 where $\eta,\delta_1\leq C(n,p,\Lambda,\gamma_0,\varepsilon,[w]_{\mathbf{A}_\infty})$ and 
	$\delta_2\leq C(n,p,\Lambda,\gamma_0,\varepsilon,[w]_{\mathbf{A}_\infty}, R/R_0)$.

With \eqref{5hh2310131} and \eqref{2nd-check} in hand,	we can now apply Lemma \ref{5hhvitali2} with $E= E_{\lambda, \delta_2}$ and $F=F_\lambda$ to complete the proof of the theorem.                         
\end{proof}

\section{Proofs of Theorem \ref{101120143-p} and Corollary \ref{compactness}}\label{sec-4}

In this section we prove Theorem \ref{101120143-p} and Corollary \ref{compactness}.   We mention here that the proof of the weighted Lorentz space bound, Theorem 
 \ref{101120143}, can be done similarly to that of Theorem \ref{101120143-p} and thus will be skipped. 
We now begin with the proof of Theorem \ref{101120143-p} using mainly the good-$\lambda$ type bound
obtained in Theorem  \ref{5hh23101312}.

\medskip
 
\noindent \begin{proof}[Proof of Theorem \ref{101120143-p}] 
	By Theorem \ref{5hh23101312}, for any $\varepsilon>0,R_0>0$ one finds  $\delta=\delta(n,p,\Lambda,\varepsilon,[w]_{{\bf A}_\infty})\in (0,1/2)$, $\delta_2=\delta_2(n,p,\Lambda,\varepsilon,[w]_{{\bf A}_\infty},diam(\Omega)/R_0)\in (0,1)$, and $\Lambda_0=\Lambda_0(n,p,\gamma_0,\Lambda)>1$ such that if $\Omega$ is  a $(\delta,R_0)$-Reifenberg flat domain and $[A]_{R_0}\le \delta$ then 
	\begin{align*}
	& w(\{({\bf M}(|\nabla u|^{\gamma_0}))^{1/\gamma_0}>\Lambda_0\lambda, (\mathbf{M}_1(\mu))^{\frac{1}{p-1}}\le \delta_2\lambda \}\cap \Omega)\\
	& \qquad \qquad \leq C\varepsilon w(\{ ({\bf M}(|\nabla u|^{\gamma_0}))^{1/\gamma_0}> \lambda\}\cap \Omega),
	\end{align*}
	for all $\lambda>0$. Here the constant $\gamma_0$ is as in Lemmas  \ref{111120149},  and the constant $C$  depends only on $n,p,\gamma_0,\Lambda,[w]_{{\bf A}_\infty}$, and $diam(\Omega)/R_0$. 
Thus, as $\Phi$ is invertible with $\Phi^{-1}:[0,\infty)\rightarrow [0,\infty)$, we find 	
\begin{align*}
	&w(\{({\bf M}(|\nabla u|^{\gamma_0}))^{1/\gamma_0}>\Phi^{-1}(t)\}\cap \Omega)\\&~~~\le 	w(\{ (\mathbf{M}_1(\mu))^{\frac{1}{p-1}}> \frac{\delta_2}{\Lambda_0} \Phi^{-1}(t) \}\cap \Omega)+ C\varepsilon w(\{ ({\bf M}(|\nabla u|^{\gamma_0}))^{1/\gamma_0}> \frac{\Phi^{-1}(t)}{\Lambda_0}\}\cap \Omega)
\end{align*}
for all $t>0$. This gives, for any $T>0$,
	\begin{align*}
	&\int_{0}^{T}w(\{x\in\Omega:	\Phi[({\bf M}(|\nabla u|^{\gamma_0}))^{\frac{1}{\gamma_0}}]>t\}) dt\\&\qquad\leq C\varepsilon \int_{0}^{T}w(\{x\in\Omega:	\Phi[\Lambda_0({\bf M}(|\nabla u|^{\gamma_0}))^{\frac{1}{\gamma_0}}]>t\}) dt \\&\qquad\qquad+   \int_{0}^{T} w(\{x\in\Omega: \Phi[\frac{\Lambda_0}{\delta_2}(\mathbf{M}_1(\mu))^{\frac{1}{p-1}}]> t  \})dt.
	\end{align*}

As $\Phi(2t)\leq c\, \Phi(t)$ and $\Phi$ is increasing, this yields
\begin{align*}
	&\int_{0}^{T}w(\{x\in\Omega:	\Phi[({\bf M}(|\nabla u|^{\gamma_0}))^{\frac{1}{\gamma_0}}]>t\}) dt\\&\qquad\leq C\varepsilon \int_{0}^{T}w(\{x\in\Omega:	 H_1\Phi[({\bf M}(|\nabla u|^{\gamma_0}))^{\frac{1}{\gamma_0}}]>t\}) dt \\&\qquad\qquad+   \int_{0}^{T} w(\{x\in\Omega: H_2\Phi[(\mathbf{M}_1(\mu))^{\frac{1}{p-1}}]> t  \})dt,
	\end{align*}
where $H_1=c^{[[\log_{2}(\Lambda_0)]]}$ and  $H_2=c^{[[\log_{2}(\frac{\Lambda_0}{\delta_2})]]}$. Here $[[a]]$ denotes the smallest integer  greater than or equal to $a$.  Thus by simple changes of variables we arrive at 
\begin{align*}
	&\int_{0}^{T}w(\{x\in\Omega:	\Phi[({\bf M}(|\nabla u|^{\gamma_0}))^{\frac{1}{\gamma_0}}]>t\}) dt\\&\qquad\leq H_1 C\varepsilon \int_{0}^{\frac{T}{H_1}}w(\{x\in\Omega:	 \Phi[({\bf M}(|\nabla u|^{\gamma_0}))^{\frac{1}{\gamma_0}}]>s\}) ds \\&\qquad\qquad+   H_2 \int_{0}^{\frac{T}{H_2}} w(\{x\in\Omega: \Phi[(\mathbf{M}_1(\mu))^{\frac{1}{p-1}}]> s  \})ds.
	\end{align*}

Now using $H_1>1$ and letting $\varepsilon=\frac{1}{2H_1 C}$ we can absorb the first term on the right to the left, which yields
\begin{align*}
	&\int_{0}^{T}w(\{x\in\Omega:	\Phi[({\bf M}(|\nabla u|^{\gamma_0}))^{\frac{1}{\gamma_0}}]>t\}) dt\\
	&\qquad \qquad \leq 	2 H_2 \int_{0}^{\frac{T}{H_2}} w(\{x\in\Omega: \Phi[(\mathbf{M}_1(\mu))^{\frac{1}{p-1}}]> s  \})ds.
	\end{align*}

Then sending $T\rightarrow \infty$ in the above bound and recalling that
\begin{align*}
\int_{\Omega}\Phi(|f|) w dx= \int_{0}^{\infty}w(\{x\in\Omega:\Phi(|f(x)|)>t\}) dt,
	\end{align*}	
we deduce 
\begin{equation*}
\int_{\Omega}\Phi[({\bf M}(|\nabla u|^{\gamma_0}))^{\frac{1}{\gamma_0}}] w dx\leq 2H_2 \int_{\Omega}\Phi[(\mathbf{M}_1(\mu))^{\frac{1}{p-1}}] w dx.
\end{equation*}

This yields \eqref{101120144} as desired and  completes the proof of the theorem.
\end{proof}

\medskip

We next  prove Corollary \ref{compactness} which provides a compactness criterion for solution sets of equation \eqref{5hh070120148}:

\medskip
 
\noindent \begin{proof}[Proof of Corollary \ref{compactness}] By de la Vall\'ee-Poussin Lemma on equi-integrability, there exists  an increasing   function 
$G:[0,\infty)\rightarrow [0,\infty)$ with $G(0)=0$ and $$\lim_{t\rightarrow \infty} \frac{G(t)}{t}=\infty,$$
such that 
\begin{equation*}
 \sup_{j} \int_{\Omega} G([{\bf M}_1(|\mu_j|)]^{\frac{q}{p-1}}) w dx \leq C.
\end{equation*}
Moreover, we may assume that $G$ satisfies a moderate growth condition (see \cite{Mey}): there exists $c_1>1$ such that 
$$G(2t)\leq c_1\, G(t)\qquad \forall t\geq 0.$$ 

Then applying Theorem \ref{101120143-p} with $\Phi(t):=G(t^q)$, which is also of moderate growth, we get
     \begin{equation*}
\int_{\Omega}G(|\nabla u_j|^q) w(x)dx \leq C \int_{\Omega} G([\mathbf{M}_1(\mu_j)]^{\frac{q}{p-1}}) w(x)dx\leq C.
    \end{equation*} 

Thus by de la Vall\'ee-Poussin Lemma the set   $\{|\nabla u_j|^q\}_j$ is  also bounded and equi-integrable in   $L^1_{w}(\Omega)$.

On the other hand, from the assumption we have  
\begin{equation}\label{rhsbound}
|\mu_j|(\Omega)\leq C.
\end{equation}
 
By \eqref{rhsbound}, it follows from the proof of Theorem 3.4 in \cite{11DMOP} that there exists a subsequence $\{u_{j'}\}_{j'}$ converging a.e. to  a  function $u$ such that $|u|<\infty$ a.e.,  $T_{k}(u)\in W^{1, p}_0(\Omega)$ for all $k>0$, and  moreover
\begin{equation*}
\nabla u_{j'} \rightarrow \nabla u \quad {\rm a.e.~ in~} \Omega.
\end{equation*}

We can now apply  Vitali Convergence Theorem to obtain the strong convergence \eqref{strong-q-w}. This completes the proof of the corollary. 
\end{proof}  

\section{Existence of solutions to Riccati type equations}\label{sec-5}
 We shall prove Theorem \ref{main-Ric} in this section. To that end, we need some preliminaries. 

\begin{definition} Given  $s >1$ we define the space $M^{1, s}(\Om)$ to be the set of all finite signed measures
$\mu$  in $\Om$ such that the quantity $[\mu]_{M^{1, s}(\Om)}<+\infty$, where
$$[\mu]_{M^{1, s}(\Om)}:=\sup\left\{ |\mu|(K)/{\rm Cap}_{1,\, s}(K): {\rm Cap}_{1, \, s}(K)>0 \right\},$$
with the supremum being taken over all compact sets $K\subset\Omega$.
\end{definition}

Due to the capacitability of Borel sets, $[\mu]_{M^{1, s}(\Om)}$ remains unchanged  in the above definition
even if the  supremum is taken over all Borel sets $K\subset\Om$.

Given a nonnegative locally finite measure $\nu$ in $\RR^n$, we define its 
first order Riesz's potentials  by
\begin{equation*}
{\rm \bf I}_{1}^{\rho}\nu (x)=\int_{0}^{\rho}\frac{\nu(B_{t}(x))}{t^{n-1}} \frac{dt}{t},\qquad x\in \RR^n,
\end{equation*}
where $\rho\in (0, \infty]$. When $\rho=\infty$, we write ${\rm \bf I}_{1}\nu$ instead of ${\rm \bf I}_{1}^{\infty}\nu$ and note that in this case we have 
$${\rm \bf I}_{1}\nu(x)=c(n) \int_{\RR^n} \frac{1}{|x-y|^{n-1}} d\nu(y) ,\qquad x\in \RR^n. $$

Let $M=M(n)\geq 1$ be a constant such that 
\begin{equation}\label{A1pot}
{\bf M}({\bf I}_1(f))(x)   \leq M\, {\bf I}_{1}(f)(x), \qquad  x\in\RR^n, 
\end{equation}
for all $f\in L^1_{\rm loc}(\RR^n)$, $f\geq 0$. Recall that ${\bf M}$ is the Hardy-Littlewood maximal function.
Inequality \eqref{A1pot} follows from an application of Fubini's Theorem and the fact the function $x\mapsto |x|^{1-n}$ is an ${\bf A}_1$ weight. By an ${\bf A}_1$ weight we mean  a nonnegative function $w\in L^{1}_{\rm loc}(\RR^n)$, $w\not\equiv0$, such that   
\begin{equation*}
{\bf M}(w)(x) \leq  C\, w(x), \qquad {\rm a.e.~} x\in\RR^n,
\end{equation*}
for a constant  $C>0$. The least possible value of $C$  will be denoted by $[w]_{{\bf A}_1}$ and is called the ${\bf A}_1$ constant of $w$. It is well-known that ${\bf A}_1\subset {\bf A}_\infty$.

With $R= {\rm diam}(\Om)$ and $q >p- 1$, for each measure $\mu\in M^{1, \frac{q}{q-p+1}}(\Om)$ we  define the set
\[
\begin{split}
E_{1}(\mu):=\Big\{ v\in W_0^{1,q}(\Om): &\int_{\Om} |\nabla v|^q w dx  \leq T_1 \int_{\Om} {\bf I}_1^{2R}(|\mu|)^{\frac{q}{p-1}} wdx \\ 
 & \text{for all } w\in {\bf A}_1\cap L^\infty(\Om) \text{ such that } [w]_{{\bf A}_1}\leq M \Big\}.
\end{split}
\]
Here $T_1>0$ is to be determined. Now if $q\geq 1$, then under the strong topology of $W_0^{1,q}(\Om)$, we have that $E_1(\mu)$ is closed and convex. Note that by \cite[Corollary 2.5]{Ph1} (see also \cite[Theorem 1.2]{MV}) if $\mu\in M^{1, \frac{q}{q-p+1}}(\Om)$
then so is ${\bf I}_1^{2R}(|\mu|)^{\frac{q}{p-1}}$ with 
\begin{equation}\label{MaVer}
[{\bf I}_1^{2R}(|\mu|)^{\frac{q}{p-1}}]_{M^{1,\frac{q}{q-p+1}}} \leq C(n,p,q,R)\, [\mu]_{M^{1,\frac{q}{q-p+1}}}^{\frac{q}{p-1}}.
\end{equation}

We will need the following lemma.
\begin{lemma}\label{pwforv-lem}
Let $\frac{3n-2}{2n-1}<p\leq 2-\frac{1}{n}$, $q>p- 1$, $\mu\in M^{1, \frac{q}{q-p+1}}(\Om)$. Let  $E_{1}(\mu)$  and $T_1>0$ be as above. Then for any $v\in E_1(\mu)$ we have 
\begin{equation}\label{pwforv}
{\bf I}^{2R}_1(|\nabla v|^q \chi_\Om)(x) \leq C_1\, T_1\, [\mu]^{\frac{q-p+1}{p-1}}_{M^{1,\frac{q}{q-p+1}}(\Om)}\  {\bf I}_{1}^{2R}(|\mu|)(x),
\end{equation}
for a.e. $x\in\Om$ and  $R= {\rm diam}(\Om)$. Here $C_1$ depends only on $n, p,$ and $q$.
\end{lemma}
\begin{proof}
For any $v\in E_1(\mu)$, by \eqref{A1pot} we have 
$$\int_{\RR^n}|\nabla v|^q \chi_\Om {\bf I}_1(f) dx \leq T_1 \int_{\RR^n} {\bf I}^{2R}_1(|\mu|)^{\frac{q}{p-1}}\chi_\Om {\bf I}_1(f) dx,$$
for any $f\in L^\infty(\RR^n), f\geq 0$, with compact support. Hence by Fubini's  Theorem, 
$$\int_{\RR^n} {\bf I}_1(|\nabla v|^q \chi_\Om) f dx \leq T_1 \int_{\RR^n} {\bf I}_1[{\bf I}^{2R}_1(|\mu|)^{\frac{q}{p-1}}\chi_\Om]f dx,$$
which yields
$${\bf I}_1(|\nabla v|^q \chi_\Om) \leq T_1\, {\bf I}_1[{\bf I}^{2R}_1(|\mu|)^{\frac{q}{p-1}} \chi_{\Om}]
\leq C\, T_1\, {\bf I}_1^{2R}[{\bf I}^{2R}_1(|\mu|)^{\frac{q}{p-1}} \chi_{\Om}]\quad {\rm a.e. ~in~} \Om.$$

On the other hand, by inequality (2.10) of \cite{Ph1} with $\nu=|\mu|, \rho=4R, \alpha=1/2, p=2$ and with $q$ replaced by $\frac{q}{p-1}>1$ 
we find 
\begin{equation*}
{\bf I}_{1}^{2R}[{\bf I}_{1}^{2R}(|\mu|)^{\frac{q}{p-1}}](x)\leq C [\mu]^{\frac{q-p+1}{p-1}}_{M^{1,\frac{q}{q-p+1}}(\Om)}\  {\bf I}_{1}^{4R}(|\mu|)(x)\leq C [\mu]^{\frac{q-p+1}{p-1}}_{M^{1,\frac{q}{q-p+1}}(\Om)}\  {\bf I}_{1}^{2R}(|\mu|)(x).
\end{equation*}
for a.e. $x\in\Om$. Thus  for a.e. $x\in\Om$ we have \eqref{pwforv} as desired.
\end{proof}

\medskip
We are now ready to prove Theorem \ref{main-Ric}.

\medskip

\noindent \begin{proof}[Proof of Theorem \ref{main-Ric}]
Since $q\geq1$, we have $q>p-1$. First we assume that $\mu\in \mathfrak{M}_0(\Om)$ and  let $S: E_1(\mu)\rightarrow W_0^{1,q}(\Om)$ be defined by 
$S(v)=u$ where $u\in W_0^{1,q}(\Om)$ is the unique renormalized  solution of
\begin{eqnarray*}
\left\{\begin{array}{rcl}
-{\rm div}(A(x, \nabla u)) &=& |\nabla v|^q + \mu \quad {\rm in}~ \Om,\\
u&=&0\quad {\rm on}~ \partial\Om.
\end{array}
\right.
\end{eqnarray*}

We  claim that we can find $T_1>0$ and $c_0>0$ such that  if (1.7) holds with $c_0$ then
\begin{equation}\label{claim-self}
S: E_1(\mu) \rightarrow E_1(\mu).
\end{equation}
 Indeed for any weight $w\in {\bf A}_1\cap L^\infty(\Om)$ with $[w]_{{\bf A}_1}\leq M$ and for any $v\in E_1(\mu)$, by Theorem \ref{101120143-p} 
there exists $\delta=\delta(n, p, \Lambda, q)\in (0, 1)$ such that if  $[A]_{R_0}\leq \delta$ and $\Om$ is $(\delta, R_0)$-Reifenberg flat for some $R_0>0$, then
we have
\begin{eqnarray*}
\int_{\Om} |\nabla S(v)|^q w dx \leq N\int_\Om \Big[{\bf I}_1^{2R}(|\nabla v|^q\chi_{\Omega} +|\mu|)\Big]^{\frac{q}{p-1}} w dx.
\end{eqnarray*}
Here $N>0$  depends only on $n, p, q, \Lambda,$ and  $R/R_{0}$. We used the elementary fact that, for any $\nu\in \mathfrak{M}_b(\Omega)$, 
\begin{equation}\label{MI}
{\bf M}_1(|\nu|)\leq C\, {\bf I}_1^{2R}(|\nu|) \quad \text{a.e. in } \Omega.
\end{equation}
 
Thus by 
Lemma \ref{pwforv-lem} we find
\begin{equation*}
\int_{\Om} |\nabla S(v)|^q w dx \leq N \Big[C_1 T_1 [\mu]^{\frac{q-p+1}{p-1}}_{M^{1,\frac{q}{q-p+1}}(\Om)}  +1\Big]^{\frac{q}{p-1}}\int_\Om {\bf I}_1^{2R}(|\mu|)^{\frac{q}{p-1}} w dx.
\end{equation*}

We now choose $T_1=2 N$ and  $c_0$ such that 
\begin{equation}\label{ci1}
0<c_0 \leq c_1:= [(2^{\frac{p-1}{q}}-1)/(2 N C_1)]^{\frac{p-1}{q-p+1}}.
\end{equation}
Then it follows that if condition (1.7) holds then
\begin{equation*}
\int_{\Om} |\nabla S(v)|^q w dx \leq T_1 \int_\Om {\bf I}_1^{2R}(|\mu|)^{\frac{q}{p-1}} w dx.
\end{equation*}
 
This gives \eqref{claim-self} for this choice of $T_1$ and $c_0$.

We next  show the continuity of $S$  on  $E_1(\mu)$. Let $\{v_k\}$ be a sequence in $E_1(\mu)$ such that $v_k$ converges strongly in  $W_{0}^{1, q}(\Om)$
to a function $v\in E_1(\mu)$. Set $u_k=S(v_k)$. We have
\begin{eqnarray*} 
\left\{\begin{array}{rcl}
-{\rm div}(A(x, \nabla u_k)) &=& |\nabla v_k|^q + \mu \quad {\rm in}~ \Om,\\
u_k&=&0\quad {\rm on}~ \partial\Om.
\end{array}
\right.
\end{eqnarray*}

Now it follows from Lemma \ref{pwforv-lem} that 
\begin{equation*}
 {\bf I}^{2R}_1(|\nabla v_k|^q\chi_{\Om} +|\mu|)(x) \leq C\, {\bf I}^{2R}_1(|\mu|)(x) \quad {\rm a.e.~} x\in\Om.
\end{equation*}
Thus by \eqref{MI} and Corollary \ref{compactness}
there exist a subsequence 
$\{u_{k'}\}_{k'}$ and a finite a.e. function $u$ with the property that $T_s(u)\in W^{1,p}_0(\Omega)$ for all $s>0$, $u_{k'}\rightarrow u$ a.e., and 
\begin{equation}\label{stronguk'u}
\nabla u_{k'} \rightarrow \nabla u \quad \text{strongly in} \quad L^q(\Omega, \mathbb{R}^n).
\end{equation}

Since $u_k\in W^{1,q}_0(\Om)$, we also have $u \in W^{1,q}_0(\Om)$. Moreover, as $|\nabla v_k|^q \rightarrow |\nabla v|^q$ in $L^1(\Om)$, by the stability result in  [8, Theorem 3.4], we have   $u=S(v)\in W^{1,q}_0(\Om)$. 
Thus the limit  in \eqref{stronguk'u} is independent of the subsequence, which implies that 
the whole sequence  $u_{k}\rightarrow u$ strongly in   $W^{1,q}_0(\Om)$. This proves the continuity of $S$.

Similarly, using Corollary \ref{compactness} we can show that $S(E_1(\mu))$ is precompact under the strong topology of $W_{0}^{1, q}(\Om)$.

At this point, we can apply Schauder Fixed Point Theorem to obtain a renormalized solution $u\in E_1(\mu)$ to equation 
\eqref{Riccati}. Moreover, by Lemma \ref{pwforv-lem},  Theorem \ref{101120143-p}, and \eqref{MI},  we have 
\begin{eqnarray*}
\int_{\Om} |\nabla u|^q w dx \leq C \int_\Om {\bf I}_1^{2R}(|\mu|)^{\frac{q}{p-1}} w dx
\end{eqnarray*}
for all weights $w\in {\bf A}_{\infty}$.Thus by \eqref{MaVer} and \cite[Lemma 3.1]{MV}, we obtain
\begin{equation*} 
[|\nabla u|^q]_{M^{1, \frac{q}{q-p+1}}(\Om)} \leq  C [\mu]_{M^{1, \frac{q}{q-p+1}}(\Om)}^{\frac{q}{p-1}}\leq C.
\end{equation*}

This completes the proof of Theorem 1.6 in the case  $\mu\in \mathfrak{M}_0(\Om)$.

We now  remove  the assumption $\mu\in \mathfrak{M}_0(\Om)$. Recall that $\mu=\mu_0+\mu_s$ where
$\mu_0\in \mathfrak{M}_{0}(\Om)$ and  $\mu_s\in \mathfrak{M}_{s}(\Om)$. Let $\mu_k=\mu_0 +\rho_k*\mu_s$, where $\{\rho_k\}_{k>0}$ is a standard sequence of mollifiers. Then by \cite[Lemma 5.7]{55Ph2}, we have $\mu_k \in \mathfrak{M}_{0}(\Om)\cap
M^{1,\frac{q}{q-p+1}}(\Om)$ with
$$[\mu_k]_{M^{1,\frac{q}{q-p+1}}(\Om)} \leq B\, [\mu]_{M^{1,\frac{q}{q-p+1}}(\Om)}$$
for some $B>1$. Thus if we further restrict $c_0$ so that $B\,c_0 \leq c_1$, where  $c_1$ is defined in \eqref{ci1}, then
 we have   
$$[\mu_k]_{M^{1,\frac{q}{q-p+1}}(\Om)} \leq c_1.$$

This allows us to apply the above result: for each $k>0$ there exists a renormalized solution $u_k\in E_1(\mu_k)$ to the equation
\begin{eqnarray*}
\left\{\begin{array}{rcl}
-{\rm div}(A(x, \nabla u_k)) &=& |\nabla u_k|^q + \mu_k \quad {\rm in}~ \Om,\\
u_k&=&0\quad {\rm on}~ \partial\Om,
\end{array}
\right.
\end{eqnarray*}
 such that 
\begin{equation*}
[|\nabla u_k|^q]_{M^{1, \frac{q}{q-p+1}}(\Om)} \leq  C 
\end{equation*}
and
\begin{equation*}
{\bf I}^{2R}_1(|\nabla u_k|^q \chi_{\Om} +|\mu_k|)(x) \leq C \  {\bf I}_{1}^{2R}(|\mu_k|)(x) \qquad {\rm a.e.~} x\in\Om.
\end{equation*}

Thus by \eqref{MI},
\begin{equation*}
{\bf M}_1(|\nabla u_k|^q \chi_{\Om} +|\mu_k|)(x) \leq C \  {\bf I}_{1}^{2R}(|\mu_k|)(x) \qquad {\rm a.e.~} x\in\Om.
\end{equation*}

Now observe that we have   
\begin{align*}
  {\bf I}_{1}^{2R}(|\mu_k|) &\leq {\bf I}_1^{2R}(|\mu_0|)+ {\bf I}_1^{2R}(\rho_k*|\mu_s|)\\ 
	&= {\bf I}_1^{2R}(|\mu_0|)+\rho_k* ({\bf I}_1^{2R}(|\mu_s|)) \\
	&\leq 	2{\bf M} ({\bf I}_1^{2R}(|\mu|)) \qquad {\rm a.e.~} x\in\Om,
\end{align*}
where  ${\bf M}$ is the Hardy-Littlewood maximal function.

Thus, since ${\bf M} ({\bf I}_1^{2R}(|\mu|))\in L^{\frac{q}{p-1}}(\Om)$, we see that 
the set $\{{\bf M}_1(|\nabla u_k|^q +\mu_k)^{\frac{q}{p-1}}\}$ is equi-integrable in $L^1(\Om)$. Then by Corollary \ref{compactness}
and the stability result of [8, Theorem 3.4], there exists a subsequence $\{u_{k'}\}$ converging a.e. and strongly in 
$W^{1, q}_0(\Om)$ to  a function $u\in W^{1, q}_0(\Om)$ such that $u$ solves \eqref{Riccati}. This completes the proof of the theorem.
\end{proof}\medskip\\

\noindent {\large \bf Acknowledgment:}

\noindent  Q.-H. Nguyen  is supported by  the Centro De Giorgi, Scuola Normale Superiore, Pisa, Italy. 
N. C. Phuc is supported in part by Simons Foundation, award number 426071.

\end{document}